\documentclass[12pt]{article}
\usepackage{graphics}
\usepackage[dvips]{graphicx}

\usepackage{amsmath,amssymb,amsthm}
\usepackage[utf8]{inputenc}
\usepackage{tikz}
\usepackage{subfigure}

\title{Star-shaped trajectories of certain billiards around a triangle}

\author{Takeo Noda, Shin-ichi Yasutomi and Masamichi Yoshida}

\newtheorem{thm}{Theorem}[section]

\newtheorem{lem}[thm]{Lemma}
\newtheorem{prop}[thm]{Proposition}
\newtheorem{con}[thm]{Conjecture}

\theoremstyle{definition}

\newtheorem{example}{Example}[section]

\theoremstyle{remark}

\numberwithin{equation}{section}

\usepackage{color}

\begin{document}

\maketitle

\begin{abstract}
We explore the triangle outer billiards map in points  at infinity
in the hyperbolic plane, focusing on the rotation number. 
Building on Dogru and Tabachnikov’s work, which established the conditions for triangles where the rotation number of the billiard map is $1/3$, we examine cases where the rotation number is $2/5$. We provide a sufficient condition for this rotation number and show its necessity for large isosceles triangles. The results are framed within the context of the Beltrami-Klein model. We concludes with a conjecture based on the findings.
\end{abstract}

\section{Introduction}
Dogru and Tabachnikov \cite{DT} studied the polygonal outer billiards map in the hyperbolic plane, also referred to as the dual billiard. For more on dual billiards, see Chapter 9 in \cite{T}. Let $C$ be a convex  in the hyperbolic plane. From a point $v$ outside $C$, two tangent lines can be drawn, intersecting $C$ at points $u_1$ and $u_2$. Assume that the direction of $vu_2$ is counterclockwise relative to $vu_1$, with an angle less than $\pi$. Define $F(v)$ as the point $w$ on the extension of $vu_1$ such that $|vu_1| = |u_1w|$. By taking limits, $F$ extends naturally to the set of infinity points, where it is denoted as $f$ (see Figure \ref{billiard2}).
Let $\rho(C) \in (0, 1/2)$ represent the rotation number of $f$. In \cite{DT}, a class of $n$-gons ($n\geq 3$) is identified where the rotation number is $1/n$, and $f$ has a hyperbolic $n$-periodic orbit. These polygons are termed "large." They also provided a specific condition for $C$ to be large, particularly when $C$ is a triangle.

\begin{thm}[Dogru and Tabachnikov \cite{DT}]\label{Dogru and Tabachnikov}
    $\triangle P_1P_2P_3$ is large if and only if $H > 1$ where $H$ is given by any of the following equal expressions: for $i=1,2,3$
\begin{align*}
H&=\sinh h_i \sinh a_i=\sin \alpha_i \sinh a_{i+1} \sinh a_{i+2}=2\sinh s \tanh r=\\
&=2\sqrt{\sinh  s \sinh (s-a_1) \sinh (s-a_2) \sinh (s-a_3)}=\\
&=4\sin(K/2)\cosh(a_1/2)\cosh(a_2/2)\cosh(a_3/2),
\end{align*}
where  $a_i$ is the length of the side corresponding to the point $P_i$,
$s=(a_1+a_2+a_3)/2$, $\alpha_i$ is the angle of $P_i$, $h_i$ is the altitude
dropped on the $i$-th side, $K$ is  the area, $r$ is the radius of the inscribed circle and the index $i$ is cyclic with respect to $\mod 3$.
\end{thm}
We note that in Theorem \ref{Dogru and Tabachnikov}, the rotation number equals $1/3$ when the condition $H \geq 1$ holds, with equality occurring when $f$ has a unique $3$-periodic orbit.

In a more generalized setting, Mozgawa \cite{M} explored a similar billiard problem and proved a theorem analogous to Poncelet’s porism,
using two ovals instead of circles. Mozgawa et al. \cite{CMM}, \cite{M} referred to this as a bar billiard.
 Cima et al. \cite{CGM} examined a bar billiard problem where the unit circle is inside
 the curve ${x^{2m}+y^{2m}=2},\ m\in {\mathbb{Z}_{>0}}$,
 and found that the rotation number is $\frac{1}{4}$ and the associated map is not conjugate to a rotation,
 except when $m=1$. Theorem \ref{Dogru and Tabachnikov} can also be viewed as addressing a bar billiard problem (see Figure \ref{Bar billiard})
 involving a circle and a triangle.
Cie\'{s}lak, Martini, and Mozgawa \cite{CMM}		
also provided the condition that the rotation number of $f$ is $2/5$, when a circle is placed inside the unit circle. 
In this paper, we provide a sufficient  condition to ensure  that the rotation number is $2/5$.
More precisely, if the distance between one vertex of a triangle and the line connecting the other two vertices falls within a certain range, then the rotation number attributed to the bar billiard is 2/5. In this context, the distance between two points is defined in relation to the Beltrami-Klein model.
For certain isosceles triangles, we show that the condition is a necessary condition.
We remark that if the rotation number attributed with  the
bar billiard is $2/5$, for a certain initial point the trajectory of  the
bar billiard is a pentagram.
The paper is organized as follows.
In Section 2, we discuss cases where a polygon is placed inside the unit circle, and we present some fundamental properties of this problem. In Section 3, we introduce the notations related to the Beltrami-Klein model, which play a crucial role in describing our results. 
In Section 4, we present our main result, which provides a sufficient condition for the rotation number to be $2/5$. In Section 5, we present additional results. Finally, in Section 6, we provide a necessary and sufficient condition for the rotation number to be $2/5$ for sufficiently large isosceles triangles.
Finally, in Section 7, we will present a conjecture.

\begin{figure}
    \centering
\begin{tikzpicture}
\draw(0,0) circle (2);
\draw [](0.2*2, 0.1*2)--(-0.4*2, 0.44641*2)--(-0.4*2, -0.24641*2)--cycle;
\draw[->] (0.984166*2, -0.177249*2)--(-0.784777*2, 0.619778*2);
);

 \draw (0.984166*2, -0.177249*2)node[right]{$v$};
 \draw (-0.6, 0.3)node[right]{$C$};

 \draw (-0.784777*2+0.1, 0.619778*2)node[left]{$f(v)$};
 \end{tikzpicture}
    \caption{Outer billiard map}
    \label{billiard2}
\end{figure}

\begin{figure}
    \centering
\begin{tikzpicture}
\draw plot[variable=\t, smooth] ({0.6*cos(\t r)+0.5*(0.6*sin(\t r)+0.4)},{0.6*sin(\t r)+0.4});

\draw[->] (2*1, 0)--(-0.256971, 0.837107*2);
\draw[->] (-0.547039*2+0.837107, 0.837107*2)--(0.109269*2-0.994012, 2*-0.994012);
\draw[->] (0.109269*2-0.994012, 2*-0.994012)--(0.419377*2+0.907812, 0.907812*2);
\draw[->] (0.419377*2+0.907812, 0.907812*2)--(-0.982513*2-0.186196, -0.186196*2);
);
\draw plot[variable=\t, smooth] ({2*cos(\t r)+1*sin(\t r)},{2*sin(\t r)});
 \draw (2*1, 0)node[right]{$v_0$};
 \draw (-0.547039*2+0.837107, 0.837107*2)node[above]{$v_1$};
\draw (0.109269*2-0.994012, 2*-0.994012)node[below]{$v_2$};
\draw (0.419377*2+0.907812, 0.907812*2)node[right]{$v_3$};
\draw (-0.982513*2-0.186196, -0.186196*2)node[left]{$v_4$};
 \end{tikzpicture}
    \caption{Bar billiard}
    \label{Bar billiard}
\end{figure}

\section{Fundamental properties}
Let $D$ be the  disk $\{(x,y)\in {\mathbb R}^2\mid x^2+y^2<1\}$.
Let $O=(0,0)$ denote the origin.
Let $U$ be a closed convex set in $D$.
For $v\in {S^1}$, we define  $\psi_{U}(v) \in {S^1}$  as the point  closest 
to $v$ in a counterclockwise direction among  the points $w$ such that $vw$ are 
tangent to $U$. 
More precisely,
we define  a map $F$ on $D\times S^1$ to $\mathbb{R}$  by
for $(u,v)\in D\times S^1$
\begin{align*}
F(u,v):=\text{the gradient of the line $uv$ with the cordinates by the basis $\{v, R(v)\}$},
\end{align*}\
where $R(v)$ is rotated at right angles to $v$.
Let us define $\theta:S^1 \to \mathbb{R}$ as
for $v\in S^1$ $\theta(v):=\min \{F(u,v)\mid u\in {U}\}$.
Since $F$ is continuous on ${U}\times S^1$, 
$\theta$ is also continuous on $S^1$.
For $v\in S^1$, $\psi_{U}(v)$ is defined as the intersection of 
$S^1$ and the line passing through $v$ with the gradient ${\theta}(v)$ by the basis $\{v, R(v)\}$.
The following lemma is straightforward.

\begin{lem}\label{cer-1}
Let $U$ be a closed convex set in $D$.
Then, $\psi_{U}$ is a homeomorphism of ${S^1}$.
\end{lem}

Let $\pi:{\mathbb R}\to {S^1}(={\mathbb R}/{\mathbb Z})$ be the projection
\begin{align*}
\pi(x):=x-\lfloor x\rfloor, \text{\ for }x\in {\mathbb R},
\end{align*}
where $\lfloor x\rfloor$ is an integral part of $x$.
For a homeomorphism $g$ of ${S^1}$, 
let $\overline{g}$ be the lift of $g$ with $\overline{g}(0) \in [0,1)$.  
For a homeomorphism $g$ of ${S^1}$ and $x\in {S^1}$,
we define $g'(x)$ as $\overline{g}'(u)$, $g'_+(x)$ as $\overline{g}_+'(u)$, and  
$g'_-(x)$ as $\overline{g}_-'(u)$ assuming each exists, where $u\in \pi^{-1}(x)$, 
$\overline{g}'(u)$ is the derivative at $u$, 
$\overline{g}_+'(u)$ is the right hand derivative at $u$,
 and $\overline{g}_-'(u)$ is the left hand derivative at $u$.

The following is derived from Lemma 5 (\cite{DT}), but we provide a proof.

\begin{lem}\label{fund-1}
Let $P$ be a point in $D$.
Then, for $v\in {S^1}$ $\psi_{P}'(v)=\dfrac{|P\psi_{P}(v)|}{|vP|}$.
\end{lem}
\begin{proof}
From the Figure \ref{derivative}, we see 
\begin{align*}
\psi_{P+}'(v)&=\lim_{x\to v^+(x\in S^1)}\dfrac{|\psi_{P}(v)\psi_{P}(x)|}{|xv|}\\
&=\lim_{x\to v^+(x\in S^1)}\dfrac{|P\psi_{P}(x)|}{|vP|}=\dfrac{|P\psi_{P}(v)|}{|vP|}.
\end{align*}
Similarly, we have $\psi_{P-}'(v)=\dfrac{|P\psi_{P}(v)|}{|vP|}$.

\begin{figure}
    \centering
\begin{tikzpicture}
\draw(0,0) circle (2);
\draw (-2*0.72413,2*0.68965)--(2,0);
\draw (1.9106,0.5910)--(-1.7377,0.99004);

 \draw (0,0.8)node[above]{$P$};
  \draw (-1.7377,0.99004)node[left]{$\psi_{P}(x)$};
  \draw (-2*0.72413,2*0.68965)node[left]{$\psi_{P}(v)$};

\draw (2,0)node[right]{$v$};
\draw (1.9106,0.5910)node[right]{$x$};
\coordinate (P) at (0,0.8);
\fill (P) circle [radius=1.5pt];

\coordinate (v) at (2,0);
\fill (v) circle [radius=1.5pt];
\coordinate (pv) at (-2*0.72413,2*0.68965);
\fill (pv) circle [radius=1.5pt];
\coordinate (x) at (1.9106,0.5910);
\fill (x) circle [radius=1.5pt];
\coordinate (px) at (-1.7377,0.99004);
\fill (px) circle [radius=1.5pt];

 \end{tikzpicture}
    \caption{$\psi_{P}$}
    \label{derivative}
\end{figure}

\end{proof}

For $v_1,v_2\in {S^1}$, we denote by $arc[v_1,v_2)$ the arc counterclockwise from $v_1$ to $v_2$ that includes $v_1$ and does not include $v_2$. We define $arc(v_1,v_2)$ similarly.

Following Lemma follows from Lemma \ref{fund-1}. 
\begin{lem}\label{fund-1-2} 
Let $P$ be a point in $D$ with $P\ne O$.
Let $u_1, u_2\in {S^1}$ be intersections of the line  $OP$ and ${S^1}$ such that
$|u_1P|<|u_2P|$.
Then, $\psi_{P}'(v)$ is  monotonic increasing for $v\in arc[u_2,u_1)$ 
and monotonic decreasing for $v\in arc[u_1,u_2)$.
\end{lem}

We present some fundamental properties of $\psi_{B}$ for a convex polygon $B$ in $D$. The following proposition follows immediately from the proof of Lemma \ref{fund-1}.

\begin{prop}\label{fund-2}
For $n\in {\mathbb{Z}_{>0}}$, let $B$ be a convex $n$ sided polygon in $D$ 
which has vertices $A_1,\ldots,A_n$ counterclockwise.
We set $A_{n+1}=A_1$.
For $1\leq i\leq n$, let $u_i, v_i\in {S^1}$ be intersections of the line  $A_{i}A_{i+1}$ and ${S^1}$ such that
$u_i$ is closer to $A_{i}$ than $A_{i+1}$ (see Figure \ref{u_i}). We set $u_{n+1}=u_1$. 
Then, 
\begin{enumerate}
\item[(1)]
if $v\in arc[u_i,u_{i+1})$ for $1\leq i \leq n$,  $\psi_{B}(v)=\psi_{A_{i+1}}(v)$,
\item[(2)]
if $v\in arc(u_i,u_{i+1})$ for $1\leq i \leq n$,  $\psi_{B}'(v)=\dfrac{|A_{i+1}\psi_{A_{i+1}}(v)|}{|vA_{i+1}|}$,
\item[(3)]
for $1\leq i\leq n$,  $\psi_{B+}'(u_i)=\dfrac{|A_{i+1}\psi_{A_{i+1}}(u_{i})|}{|u_{i}A_{i+1}|}$
 and $\psi_{B-}'(u_i)=\dfrac{|A_{i}\psi_{A_{i}}(u_i)|}{|u_iA_{i}|}$.

\end{enumerate}

\begin{figure}
    \centering
\begin{tikzpicture}
\draw (0,1)--(-1,0)--(0,-1)--cycle;
\draw(0,0) circle (2);
 \draw (0,1)node[right]{$A_1$};
  \draw (-1,0)node[left]{$A_2$};
 \draw (0,-1)node[right]{$A_3$};
\draw[dashed] (0,1)--(0.82287,1.82287);
\draw[dashed] (-1,0)--(-1.822875,0.82287);
\draw[dashed] (0,-1)--(0,-2);

\draw (0.82287,1.82287)node[right]{$u_1$};
\draw (-1.822875,0.82287)node[left]{$u_2$};
\draw (0,-2)node[below]{$u_3$};

\coordinate (A1) at (0,1);
\fill (A1) circle [radius=1.5pt];
\coordinate (A2) at (-1,0);
\fill (A2) circle [radius=1.5pt];
\coordinate (A3) at (0,-1);
\fill (A3) circle [radius=1.5pt];
\coordinate (u1) at (0.82287,1.82287);
\fill (u1) circle [radius=1.5pt];
\coordinate (u2) at (-1.822875,0.82287);
\fill (u2) circle [radius=1.5pt];
\coordinate (u3) at (0,-2);
\fill (u3) circle [radius=1.5pt];

 \end{tikzpicture}

    \caption{$u_i{}'s$ for $n=3$}
    \label{u_i}
\end{figure}
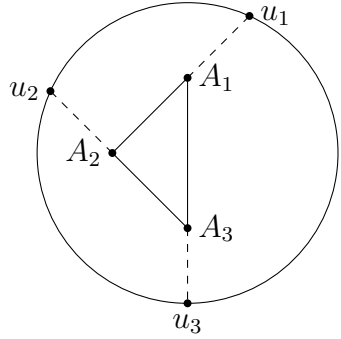

\end{prop}

For an orientation preserving homeomorphism $f$ on ${S^1}$, let $\rho(f)$ be the rotation number of $f$; i.e.,
$$
\rho(f):=\lim_{n\to \infty}\dfrac{\overline{f}^n(x)-x}{n},
$$
where $x\in {\mathbb R}$ and $\overline{f}$ is the lift of $f$ on ${\mathbb R}$
with $\overline{f}(0)\in [0,1)$.

Then, following results(\cite{KB}) are known.

\begin{lem}\label{fund-3}
Let $f$ be an orientation preserving homeomorphism  on ${S^1}$.
\begin{enumerate}
\item[(1)] $\rho(f)$ is well defined.
\item[(2)] $\rho(f)$ is a rational $\dfrac{p}{q}$, if and only if
there exits a $x\in {\mathbb R}$ such that
$\overline{f}^q(x)=x+p$, where $p$ and $q$ are relatively prime positive integers.  
\item[(3)] When $\rho$ is considered as a function on the set of  orientation preserving homeomorphisms on ${S^1}$, 
it is continuous on $C^0$ topology.
\end{enumerate}

\end{lem}

\begin{lem}\label{fund-4}
Let $U$ be a closed convex set in $D$.
\begin{enumerate}
\item[(1)] $\psi_{U}$ is the orientation preserving homeomorphism of ${S^1}$.
\item[(2)] If $V$ is a closed convex set in $D$ with $V \subset U$, then $\rho(\psi_{U})\leq \rho(\psi_{V})$.  
\item[(3)] $\rho(\psi_{U})\leq \dfrac{1}{2}$.
\end{enumerate}
\end{lem}
\begin{proof}
Let us prove (1).
From Lemma \ref{cer-1}, $\psi_{U}$ is a homeomorphism of ${S^1}$.
Let $v\in {S^1}$.
Let $v'$ be a point in ${S^1}$ sufficiently close counterclockwise from $v$.
Then, there exists $u\in U$ which is on the line $v'\psi_{U}(v')$.
$\psi_{U}(v')$ is a point in ${S^1}$ sufficiently close counterclockwise from $\psi_{u}(v)$. 
Since the gradient of the line $v\psi_{U}(v)$ is less than or equal to  the gradient of the line  $vu$ by the basis $\{v, R(v)\}$,
$\psi_{u}(v)$ is a point in ${S^1}$ sufficiently close counterclockwise from $\psi_{U}(v)$. 
Therefore, $\psi_{U}(v')$ is a point in ${S^1}$ sufficiently close counterclockwise from $\psi_{U}(v)$, 
which means that $\psi_{U}$ is the orientation preserving homeomorphism.

Next, we prove (2).
Let $v\in {S^1}$.
Since $V \subset U$, the gradient of the line $\psi_{U}(v)v$ is less than or equal to the gradient of the line $\psi_{V}(v)v$ by the basis $\{v, R(v)\}$.
This means that $\psi_{V}(u)$ lies counterclockwise from $\psi_{U}(v)$.
Therefore, we have $\overline{\psi_{U}}(v)\leq \overline{\psi_{V}}(v)$, which implies
$\rho(\psi_{U})\leq \rho(\psi_{V})$.

Next, we prove (3).
Let $P\in U$ be a point.
It is easy to see that for $P\in D$, $\rho(\psi_{P})=\frac{1}{2}$.
By (2), we have $\rho(\psi_{U})\leq \frac{1}{2}$.

\end{proof}

\begin{lem}\label{fund-5}
Let $B$ be a convex $n$ sided polygon in $D$ such that $n\geq 2$ and
$\rho(\psi_{B})$ is rational. Then, $\psi_{B}$ is not conjugated to a rotation. 
\end{lem}
\begin{proof}
Let $\rho(\psi_{B})$ be $\dfrac{p}{q}$, where $p$ and $q$ are relatively prime positive integers.
We suppose that $\psi_{B}$ is conjugated to a rotation.
From Lemma \ref{fund-3}, there exists $u\in {\mathbb R}$ such that
$\overline{\psi_{B}}^q(u)=u+p$.
Since $\psi_{B}$ is conjugated to a rotation, for every $x\in {\mathbb R}$ $\overline{\psi_{B}}^q(x)=x+p$, which 
implies that for every $v\in {S^1}$
\begin{align*}
(\psi_{B}^q)'(v)=1.
\end{align*}
We define  $u_i\in {S^1}$ for $1\leq i\leq n$ as in Lemma \ref{fund-2}.
Therefore, by the chain rule we have
\begin{align}
\psi_{B-}'(u_1)\psi_{B-}'(u^{(1)}_1)\cdots \psi_{B-}'(u^{(q-1)}_1)=1,\label{psiB}\\
\psi_{B+}'(u_1)\psi_{B+}'(u^{(1)}_1)\cdots \psi_{B+}'(u^{(q-1)}_1)=1,\label{psiB2}
\end{align}
where $u^{(i)}_1=\psi_{B}^i(u_1)$ for $1\leq i\leq q-1$.
From Lemma \ref{fund-2}, we see that $\psi_{B-}'(u_1)>\psi_{B+}'(u_1)$ and
$\psi_{B-}'(u^{(i)}_1)\geq \psi_{B+}'(u^{(i)}_1)$, which contradict that  equations (\ref{psiB}) and (\ref{psiB2}).
\end{proof}

\section{Previous result}
Let $D$ be the  disk $\{(x,y)\in {\mathbb R}^2\mid x^2+y^2<1\}$.
$D$ is a hyperbolic geometric model that is called the Beltrami–Klein model(see \cite{L}).
A line of $D$ in the hyperbolic sense is known to be a straight line segment with endpoints on 
the boundary of $D$.
A point on the boundary of $D$ is called  a point at infinity.
We define $d'(P,Q)$ as follows. Let $v_1, v_2\in {S^1}$ be two points at infinity that intersect the hyperbolic line $PQ$ (see Figure \ref{distance}). Then, we set
\begin{align*}
\dfrac{1}{2}\left|\log\left(\dfrac{|v_1Q||v_2P|}{|v_1P||v_2Q|} \right)\right|.
\end{align*}

\begin{figure}
    \centering
\begin{tikzpicture}
\draw(0,0) circle (2);
\draw (-1/2,1.732/2)--(1,0);
 \draw (-1/2,1.732/2)node[left]{$P$};
  \draw (1,0+0.1)node[right]{$Q$};

\draw[dashed] (-1/2,1.732/2)--(-1.42705098312484, 1.40125853844407);
\draw[dashed] (1,0)--(1.92705098312484,-0.535233134659635);
\draw (-1.42705098312484, 1.40125853844407)node[left]{$v_1$};
\draw (1.92705098312484,-0.535233134659635)node[right]{$v_2$};
\coordinate (P) at (-1/2,1.732/2);
\fill (P) circle [radius=1.5pt];
\coordinate (Q) at (1,0);
\fill (Q) circle [radius=1.5pt];

 \end{tikzpicture}

    \caption{$d'(P,Q)$.}
    \label{distance}
\end{figure}

For $P,Q\in D$ and $n\in {\mathbb Z}_{>0}$, we define
$\Delta_n'(P,Q):=\log \left(\dfrac{e^{nd'(P,Q)}+1}{e^{nd'(P,Q)}-1}\right)$.
We define $\delta(P,Q,R)$ as the minimum of $d'(R,S)$, where
$S$ are points on the line $PQ$. In other words, let $RS$ be the perpendicular line drawn from $R$ to the line $PQ$ in the hyperbolic sense, with intersection point $S$. Then, $\delta(P,Q,R)=d'(R,S)$. 
From \cite[Proposition 2.2]{NY},  $\delta(P,Q,R)$ is given by
\begin{align*}
\mathrm{arcsinh}\left(\dfrac{\sqrt{(-\cosh(\alpha-\gamma)+\cosh \beta)(\cosh(\alpha+\gamma)+\cosh \beta)}}{\sinh\gamma} \right),
\end{align*}
where $\alpha=d'(Q,R), \beta=d'(R,P)$ and $\gamma=d'(P,Q)$.
The following theorem, which is a result for $\rho(\psi_{\triangle PQR})=\frac{1}{3}$, is a restatement of Theorem \ref{Dogru and Tabachnikov} using $\Delta_1'$, and  we  note that $\Delta_n'$ plays an important role in this paper.

\begin{thm}[\cite{NY}]
\label{t5}
Let $\triangle PQR$ be a triangle in $D$.
Then, $\dfrac{1}{2}> \rho(\psi_{\triangle PQR})>\dfrac{1}{3}$  if $\delta(P,Q,R)<\Delta'_1(P,Q)$
and $\rho(\psi_{\triangle PQR})=\dfrac{1}{3}$  if $\delta(P,Q,R)\geq\Delta'_1(P,Q)$.
\end{thm}

The following is the example in \cite{NY}.
\begin{example}\label{example1}
Let $P=(-\frac{1}{4}, \frac{\sqrt{3}}{4})$, $Q=(-\frac{1}{4}, -\frac{\sqrt{3}}{4})$,
$R=(\frac{1}{2}, 0)$, and $S=(-\frac{1}{4}, 0)$.
We set $v_1=(-\frac{1}{4}, \frac{\sqrt{15}}{4})$, 
$v_2=(-\frac{1}{4}, -\frac{\sqrt{15}}{4})$,
$v_3=(1,0)$,
and 
$v_4=(-1,0)$, which are points at infinity.
Then, $\triangle PQR$ is an equilateral triangle in $D$. See Figure
\ref{fig:example3-1}.
Then, 
\begin{align*}
&d'(P,Q)=\dfrac{1}{2}\log\dfrac{|v_1Q||Pv_2|}{|v_2Q||Pv_1|}
=\log \dfrac{\sqrt{5}+1}{\sqrt{5}-1},\\
&\Delta'_1(P,Q)=\log\dfrac{e^{d'(P,Q)}+1}{e^{d'(P,Q)}-1}=\log \sqrt{5}.
\end{align*}
Conversely, we have
\begin{align*}
\delta(P,Q,R)=d'(R,S)=\dfrac{1}{2}\log\dfrac{|v_3S||Rv_4|}{|v_4S||Rv_3|}=\log \sqrt{5}.
\end{align*}
Therefore, $\rho(\psi_{\triangle PQR})=\dfrac{1}{3}$.

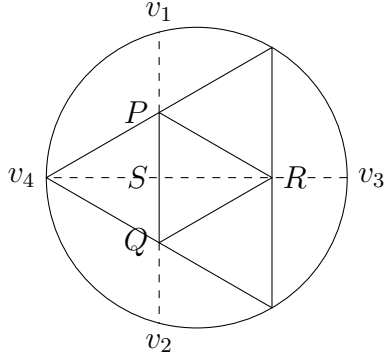
\begin{figure}
    \centering
\begin{tikzpicture}
\draw(0,0) circle (2);
\draw (-1/2,1.732/2)--(-1/2,-1.732/2)--(1,0)--cycle;
 \draw (-1/2,1.732/2)node[left]{$P$};
 \draw (-1/2,-1.732/2)node[left]{$Q$};
  \draw (1,0)node[right]{$R$};
 \draw (-1/2,0)node[left]{$S$};
 
\draw[dashed] (-1/2,1.732/2)--(-1/2,1.936491673103708);
\draw[dashed] (-1/2,-1.732/2)--(-1/2,-1.936491673103708);
\draw[dashed] (2, 0)--(-2,0);
 \draw (-1/2,1.936491673103708)node[above]{$v_1$};
\draw (-1/2,-1.936491673103708)node[below]{$v_2$};
\draw (-2, 0)node[left]{$v_4$};
\draw (2,0)node[right]{$v_3$};
\draw (1,1.732)--(1,-1.732)--(-2,0)--cycle;

 \end{tikzpicture}
    \caption{$\triangle PQR$ in Example\ref{example1}.}
    \label{fig:example3-1}
\end{figure}

\end{example}

\section{Pentagram}
This section focuses on the conditions for $\rho(\psi_{\triangle PQR})=\dfrac{2}{5}$. 
\begin{lem}\label{pentagon-1}
Let $\triangle PQR$ be a triangle in $D$.
Then, the following are equivalent:
\begin{enumerate}
\item[(1)]
There  exists a periodic point $\alpha\in {S^1}$
with $\psi_{\triangle PQR}^5(\alpha)=\alpha$.
\item[(2)]
$\rho(\psi_{\triangle PQR})=\dfrac{2}{5}$.
\item[(3)]
There  exist $5$ points $\alpha_i \in [0,1)$ 
 for $0\leq i\leq 4$ such that
 $\psi_{\triangle PQR}(\pi(\alpha_i))=\pi(\alpha_j)$ where
 $j\equiv i+2 \mod 5$ and  
 $\alpha_0<\alpha_1<\ldots <\alpha_4$ (see Figure \ref{fig:pentagon}).
\item[(4)]
 There  exists a point $\alpha\in {S^1}$
 such that for $x\in \pi^{-1}(\alpha)$, $\overline{\psi_{\triangle PQR}}^5(x)=2+x$
  \end{enumerate}
\end{lem}
\begin{proof}
(1)$\to$(2):
By virtue of Lemma \ref{fund-3}, we have $\rho(\psi_{\triangle PQR})=\frac{j}{5}$ for some $1\leq j \leq 4$. However, from Theorem \ref{t5}, we know that $\rho(\psi_{\triangle PQR})=\frac{2}{5}$.\\
(2)$\to$(1):
This follows directly from Theorem \ref{t5}.\\
(3)$\to$(1):
It is obvious.\\
(2)$\to$(3):
From Theorem \ref{t5}, There  exist $\alpha \in {\mathbb R}$ such that
 $\overline{\psi_{\triangle PQR}}^5(\alpha)=\alpha+2$.
We set $f=\overline{\psi_{\triangle PQR}}$.
We show that for $u=f^i(\alpha)+j$ for $i\in {\mathbb Z}_{\geq 0}$ and $j\in {\mathbb Z}$,
$f^5(u)=u+2$.
In fact, we have
$f^5(u)=f^{i+5}(\alpha)+j=f^{i}(\alpha)+j+2=u+2$.
We denote $f^i(\alpha)-\lfloor f^i(\alpha)\rfloor$ for $i=0,\ldots,4$ by $\alpha_j$ for $j=0,\ldots,4$ in order of its value; i.e.,
$\alpha_0<\alpha_1\ldots <\alpha_4$.
Let $I=\{\alpha_i+j\mid 0\leq j \leq 4, j\in {\mathbb Z} \}$.   
We index each element of $I$ as $\alpha_{i+5j}=\alpha_i+j$.
It is easily seen that $f(I)\subset I$ and $f^{-1}(I)\subset I$.
Since $f_I$ is monotonic increasing and bijective transformation on $I$,  
$f_I$ is considered as a shift $k>0$ transformation on $\{\alpha_{i}\mid i\in {\mathbb Z}\}$; i.e.,
$f_I(\alpha_{i})=\alpha_{i+k}$ for $i\in {\mathbb Z}$.
Since $f^5(\alpha_0)=\alpha_0+2=\alpha_{10}$, we have $k=2$, which implies (3).
From Lemma \ref{fund-3}, we have (2)$\to$(4) and (4)$\to$(2).
\end{proof}

\begin{figure}
    \centering
\begin{tikzpicture}
\draw(0,0) circle (2);
\draw (1.3264033,1.49688)--(-2,0)--(1.3264033,-1.496881)--(-0.563833,1.91887766)--(-0.5638337,-1.91887)--cycle;
\draw (0,0.9)--(-0.563833,0)--(0,-0.9)--cycle;

 \draw (0.05,0.9)node[above]{$P$};
 \draw (0.05,-0.95)node[below]{$Q$};
 \draw (-0.563833,0)node[left]{$R$};
\draw (1.3264033,1.49688)node[above]{$\pi(\alpha_0)$};
\draw (-0.563833,1.91887766)node[above]{$\pi(\alpha_1)$};
\draw (-2,0)node[left]{$\pi(\alpha_2)$};
\draw (-0.5638337,-1.91887)node[below]{$\pi(\alpha_3)$};
\draw (1.3264033,-1.49688)node[right]{$\pi(\alpha_4)$};
 \end{tikzpicture}
    \caption{Points $\pi(\alpha_i)$ for $0\leq i \leq 4$. }
    \label{fig:pentagon}
\end{figure}
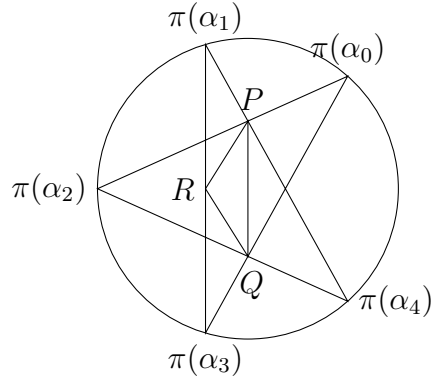

Let ${\mathbb H}$ be the upper half plane ${z \in {\mathbb C} \mid \text{the imaginary part of $z$ is positive}}$ equipped with the metric $ds^2=\frac{dx^2+dy^2}{y^2}$.
It is well-known that $H$ is isomorphic to $D$ (see \cite{L}).

\begin{lem}\label{pentagon-2}
Let $P, Q$ be different points  in $D$.
Let $v_1,v_2\in {S^1}$ be points at infinity, which are intersection points with the line $PQ$ and ${S^1}$.
Suppose that $v_1$ is closer to $P$ than $Q$.
For a point $R$ in $D$, the following statements are equivalent:
\begin{enumerate}
\item[(1)]
There exists a point $w$ at infinity with $w\ne v_i$ for $i=1,2$ such that
$R$ is the intersection of  line $v_1w_2$ and  line $v_2w_1$,
where $w_i$  $(i=1,2)$  are points at infinity with   $w\ne w_i$ such that
 $w_1$ is the intersection of  line $wP$ and $S^1$ and
$w_2$ is the intersection of  line $wQ$ and $S^1$ (see Figure \ref{fig:pentagon2}).  
\item[(2)]
$\delta(P,Q,R)=\dfrac{1}{2}\Delta_1'(P,Q)$.
\end{enumerate}
\end{lem}

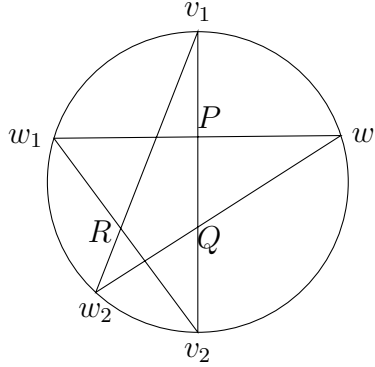
\begin{figure}
    \centering
\begin{tikzpicture}
\draw(0,0) circle (2);
\draw (1.902113,0.618033)--(-1.91348,0.581858)--(0,-2.0)--(0,2.0)--(-1.35714,-1.469062)--cycle;

 \draw (0.15,0.57)node[above]{$P$};
 \draw (0.15,-0.45)node[below]{$Q$};
 \draw (-1,-0.65)node[left]{$R$};
\draw (1.902113,0.618033)node[right]{$w$};
\draw (0,2)node[above]{$v_1$};
\draw (0,-2)node[below]{$v_2$};
\draw (-1.91348,0.581858)node[left]{$w_1$};
\draw (-1.35714,-1.469062)node[below]{$w_2$};

 \end{tikzpicture}
    \caption{Points $v_1$, $v_2$, $w$, $w_1$, and $w_2$. }
    \label{fig:pentagon2}
\end{figure}

\begin{proof}
 There exists  an isometry $f:D\to {\mathbb H}$ such that
 $f(P)=ki$ and $f(Q)=i$, where $k>1$ and $f$. The map $f$ is naturally extended to the map $\Bar{f}$ from $D\cup S^1$ to 
 ${\mathbb H}\cup {\mathbb R} \cup \{\infty\}$ such that $\Bar{f}(v_1)=\infty$ and $\Bar{f}(v_2)=0$. 
 
 \noindent
 First, we show (1)$\rightarrow$(2).
 We set $\Bar{f}(R)=u+vi$ for $u\in {\mathbb R}$ and $v\in {\mathbb R}_{>0}$.
  We assume that $\Bar{f}(w)$ is negative without loss of generality.
  We set $\Bar{f}(w)=-x$. Then, we get that  $\Bar{f}(w_1)=\dfrac{1}{x}$ and  $\Bar{f}(w_2)=\dfrac{k^2}{x}$.
 Since $\Bar{f}(R)$ is on the line between $\infty$ and $1/x$, we have $u=1/x$.
On the other hand, $\Bar{f}(R)$ lies on a circle with $0$ and $\frac{k^2}{x}$ at the two ends of its diameter.

 Therefore, we have
 \begin{align*}
 \left(\dfrac{1}{x}-\dfrac{k^2}{2x}\right)^2+v^2=\dfrac{k^4}{4x^2},
 \end{align*}
 which is equal to 
 \begin{align*}
 u=\dfrac{v}{\sqrt{k^2-1}}.
 \end{align*}
 We denote the set $\{x+yi\in {\mathbb H}\mid x=\dfrac{y}{\sqrt{k^2-1}}\}$ by $L$.
Let $\theta\in (0,\pi/2)$ be the angle between the imaginary axis and $L$.
We set $s=\tan(\theta/2)$.
Then, we have $k=\dfrac{1+s^2}{2s}$, 
 which implies 
 \begin{align}\label{k+1k-1}
 \dfrac{k+1}{k-1}=\dfrac{(1+s)^2}{(1-s)^2}.
 \end{align}
We will calculate the length $d$ of a vertical line from $u+vi$ to the imaginary axis in the hyperbolic sense. Let  $r=\sqrt{u^2+v^2}$.
Then, the length of the curve $r\sin w +ir\cos w (0\leq w\leq \theta)$ is
\begin{align}\label{int_0theta}
d=\int_0^{\theta} \dfrac{dw}{\cos w}= \log\left(\dfrac{1+\tan(\theta/2)}{1-\tan(\theta/2)}\right).
\end{align}
From (\ref{k+1k-1}) and (\ref{int_0theta}), we have
\begin{align}\label{d=fra}
d=\dfrac{1}{2}\log \left(\dfrac{k+1}{k-1} \right).
\end{align}
From the fact that the distance between $i$ and $ki$ is $\log k$ 
and (\ref{d=fra}),
we can obtain $\delta(P,Q,R)=\dfrac{1}{2}\Delta_1'(P,Q)$.
Next, we show (2)$\rightarrow$(1).
By reversing the steps of the previous proof, we can conclude that $u=\dfrac{v}{\sqrt{k^2-1}}$ or $u=-\dfrac{v}{\sqrt{k^2-1}}$.
Setting $w=\Bar{f}^{-1}(-1/u)$, $w_1=\Bar{f}^{-1}(u)$, and $w_2=\Bar{f}^{-1}(k^2u)$, we have (1).   
\end{proof}

From Lemma \ref{pentagon-2}, we can state the following lemma, which provides a sufficient condition for 
$\rho(\psi_{\triangle PQR})=\dfrac{2}{5}$.

\begin{lem}\label{pentagon-2-1}
Let $\triangle PQR$ be a triangle in $D$ with $\delta(P,Q,R)=\dfrac{1}{2}\Delta_1'(P,Q)$.
Then, $\rho(\psi_{\triangle PQR})=\dfrac{2}{5}$ holds.
\end{lem}
\begin{proof}
Applying Lemma \ref{pentagon-2} to $\triangle PQR$, we can find a point $w$ at infinity that satisfies (1) of the lemma. Using this fact, we see that $\psi_{\triangle PQR}^5(w)=w$, which implies that $\rho(\psi_{\triangle PQR})=\dfrac{2}{5}$.
\end{proof}

The next lemma states that the figure equidistant from the line in D in the hyperbolic sense is an ellipse.

\begin{lem}\label{pentagon-3}
Let $v_1=(0,1),v_2=(0,-1)\in {S^1}$. Let $k>0$.
Then, the set of $P\in D$, where the length of a vertical line from $P$ to the line $v_1v_2$ in the hyperbolic sense
is $k$ is given by
\begin{align*}
\{(x,y) \in D \mid \dfrac{(e^{2k}+1)^2}{(e^{2k}-1)^2}x^2+y^2=1\}.
\end{align*}
\end{lem}
\begin{proof}
The length of a vertical line from $P(x,y)\in D$ to the line $v_1v_2$ is the length of 
$PQ$, where $Q=(0,y)$.
Therefore, we have
\begin{align*}
\dfrac{1}{2}\left|\log\dfrac{\sqrt{1-y^2}+x}{\sqrt{1-y^2}-x}\right|=k,
\end{align*}
which is equivalent to $\dfrac{(e^{2k}+1)^2}{(e^{2k}-1)^2}x^2+y^2=1$.
\end{proof}

The following lemma explicitly gives the quintuple periodic points of $\psi_{\triangle PQR}$
when $\delta(P,Q,R)=\Delta_2'(P,Q)$.

\begin{lem}\label{pentagon-4}
Let $0<t<1$, $P=(0,t)$, and $Q=(0,-t)$.
Let $R=(u,v)\in D$ with $u\ne 0$, $|v|\leq t$ and $\delta(P,Q,R)=\Delta_2'(P,Q)$.
We set points $A_i$ $(1\leq i \leq 5)$ as follows:
\begin{align}
&A_1:=\left(\dfrac{\sqrt{1-v^2}(t^2-1)}{2vt-t^2-1},-\dfrac{vt^2+v-2t}{2vt-t^2-1}\right),\label{a1eq1}\\
&A_2:=\left(-\sqrt{1-v^2},v\right),\label{a1eq2}\\
&A_3:=\left(-\dfrac{\sqrt{1-v^2}(t^2-1)}{2vt+t^2+1},-\dfrac{vt^2+v+2t}{2vt+t^2+1}\right),\label{a1eq3}\\
&A_4:=\left(-\dfrac{\sqrt{1-v^2}(t^4-2t^2+1)}{4vt^3+t^4+4vt+6t^2+1},\dfrac{vt^4+6vt^2+4t^3+v+4t}{4vt^3+t^4+4vt+6t^2+1}\right),\label{a1eq4}\\
&A_5:=\left(\dfrac{\sqrt{1-v^2}(t^4-2t^2+1)}{4vt^3-t^4+4vt-6t^2-1},-\dfrac{vt^4+6vt^2-4t^3+v-4t}{4vt^3-t^4+4vt-6t^2-1}\right).\label{a1eq5}
\end{align}
Then, $A_i$ $(1\leq i \leq 5)$ are on ${S^1}$.
Then, if $u<0$, $\psi_{\triangle PQR}(A_i)=A_{i+1}$ for $1\leq i \leq 4$ and $\psi_{\triangle PQR}(A_5)=A_{1}$.  
if $u>0$, $\psi_{\triangle PQR}(A_{i+1})=A_{i}$ for $1\leq i \leq 4$ and $\psi_{\triangle PQR}(A_1)=A_{5}$.  

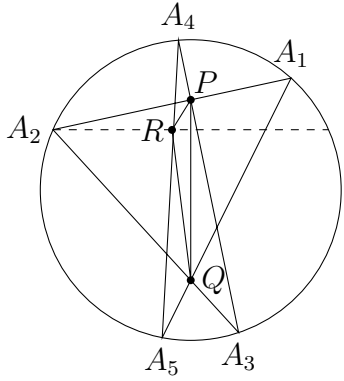
\begin{figure}
    \centering
\begin{tikzpicture}
\draw(0,0) circle (2);
\draw (1.334,1.489)--(-1.8318,0.8027)--(0.6365,-1.8959)--(-0.1631,1.9933)--(-0.379,-1.9637)--cycle;
\draw (0,1.2)--(0,-1.2)--(-0.2490,0.8)--cycle;

 \draw (1.334,1.489)node[above]{$A_1$};
 \draw (-1.8318,0.8027)node[left]{$A_2$};
 \draw (0.6365,-1.8959)node[below]{$A_3$};
\draw (-0.1631,1.9933)node[above]{$A_4$};
\draw (-0.379,-1.9637)node[below]{$A_5$};
\draw (-0.210,0.8)node[left]{$R$};
\draw (0.2,1.15)node[above]{$P$};
\draw (0,-1.2)node[right]{$Q$};
\coordinate (P) at (0,1.2);
\fill (P) circle [radius=1.5pt];
\coordinate (Q) at (0,-1.2);
\fill (Q) circle [radius=1.5pt];
\coordinate (R) at (-0.2490,0.8);
\fill (R) circle [radius=1.5pt];
\draw[dashed] (-1.8318,0.8027)--(1.8318,0.8027);

 \end{tikzpicture}
    \caption{Points $A_i$ $(1\leq i\leq 5)$. }
    \label{fig:pentagon3}
\end{figure}
\end{lem}
\begin{proof}
We assume for simplicity that $u<0$ and $v\geq 0$ (see Figure \ref{fig:pentagon3}).
We see that $A_2$ in on ${S^1}$.
The open source software SageMath \cite{Sage} is used for the following calculations.	
By finding the intersection of the line $A_2P$ with ${S^1}$ we obtain $A_1$ and its coordinates (\ref{a1eq1}).
Since the second coordinates of $A_2$ and $R$ are the same and $|v|\leq t$, we have
$\psi_{\triangle PQR}(A_1)=A_{2}$.		
By finding the intersection of the line $A_2Q$ with ${S^1}$ we obtain $A_3$ and its coordinates (\ref{a1eq3}).
Again for the same reason, we have
$\psi_{\triangle PQR}(A_2)=A_{3}$.
By finding the intersection of the line $A_3P$ with ${S^1}$ we obtain $A_4$ and its coordinates (\ref{a1eq4}).
Since the first coordinates of $A_3$ is positive and the second coordinates of $A_3$ is less than $-t$, 
we have $\psi_{\triangle PQR}(A_3)=A_{4}$.
By finding the intersection of the line $A_1Q$ with ${S^1}$ we obtain $A_5$ and its coordinates (\ref{a1eq5}).
Next we show that $R=(u,v)$ is on the line segment  $A_4A_5$. 
Since $d'(P,Q)=\log\frac{1+t}{1-t}$, we have 
$\Delta'_2(P,Q)=\log\frac{e^{2d'(P,Q)}+1}{e^{2d'(P,Q)}-1}=\log \frac{t^2+1}{2t}$.
From $\delta(P,Q,R)=\Delta_2'(P,Q)$ and Lemma \ref{pentagon-3}, we see that
$(u,v)$ satisfies that
\begin{align}\label{t^4+6t^2}
\left(\dfrac{t^4+6t^2+1}{t^4-2t^2+1}\right)^2u^2+v^2=1
\end{align}
From the formulae  (\ref{a1eq4}), (\ref{a1eq5}) and (\ref{t^4+6t^2}),  we see that $R$ in on the line $A_4A_5$ and 
by considering the second coordinates of $R, A_4$ and $A_5$ are negative,
which implies that $\psi_{\triangle PQR}(A_4)=A_{5}$.
From the facts that $|v|\leq t$ and the second coordinates of $A_{5}$ is less than $-t$, we have
$\psi_{\triangle PQR}(A_5)=A_{1}$. 
We can prove the theorem in a similar manner for the case where $u>0$.
\end{proof}

From Lemma \ref{pentagon-4}, we can state the following lemma, which provides a sufficient condition for 
$\rho(\psi_{\triangle PQR})=\dfrac{2}{5}$.

\begin{lem}\label{pentagon-5}
Let $\triangle PQR$ be a triangle in $D$.
Let $S$ be a point on the line $PQ$ such that line $RS$ is perpendicular to the line $PQ$ in the hyperbolic sense.
We suppose that $\delta(P,Q,R)=\Delta_2'(P,Q)$ and $S$ is in the line segment $PQ$.
Then, $\rho(\psi_{\triangle PQR})=\dfrac{2}{5}$ holds.
\end{lem}
\begin{proof}
Let $t=(e^{d'(P,Q)}-1)/(e^{d'(P,Q)}+1)$, $P'=(0,t)$, and $Q'=(0,-t)$.
Then, since $d'(P,Q)=d'(P',Q')$ holds, 
there exist an isometry $\phi$ on $D$
such that $\phi(P)=P'$ and $\phi(Q)=Q'$. Let $R'=\phi(R)$ and $S'=\phi(S)$.
Let $\phi'$ be an extension of $\phi$ to $D\cup {S^1}$. 
We note that $\delta(P,Q,R)=\delta(P',Q',R')$ and 
$S'$ is in the line segment $P'Q'$.
Therefore, from Lemma \ref{pentagon-4} 
there exists a periodic point $A_1'$ on ${S^1}$ with the period $5$ by the map $\psi_{\triangle P'Q'R'}$.
Since $\phi'^{-1}(A_1')$ is  a periodic point by the map $\psi_{\triangle PQR}$, 
we see $\rho(\psi_{\triangle PQR})=\dfrac{2}{5}$. 
\end{proof}

Combining Lemma \ref{pentagon-2-1} and Lemma \ref{pentagon-5}, we obtain the following result.

\begin{lem}\label{pentagon-6}
Let $\triangle PQR$ be a triangle in $D$.
Let $S$ be a point on the line $PQ$ such that line $RS$ is perpendicular to the line $PQ$ in the hyperbolic sense.
We suppose that $\Delta_2'(P,Q)\leq \delta(P,Q,R)\leq \frac{1}{2}\Delta_1'(P,Q)$ and $S$ is in the line segment $PQ$.
Then, $\rho(\psi_{\triangle PQR})=\frac{2}{5}$ holds.
\end{lem}
\begin{proof}
We take the point $R_1$ in the line $RS$ with $d'(R_1,S)=\frac{1}{2}\Delta_1'(P,Q)$ and $d'(S,R)\leq d'(S,R_1)$.
We also take the point $R_2$ in the line $RS$ with $d'(R_2,S)=\Delta_2'(P,Q)$ and $d'(S,R)\geq d'(S,R_2)$.
Since $\Delta_2'(P,Q)\leq \delta(P,Q,R)\leq \frac{1}{2}\Delta_1'(P,Q)$, 
we see $\triangle PQR_2\subset \triangle PQR\subset \triangle PQR_1$, which implies 
\begin{align*}
\rho(\psi_{\triangle PQR_2})\geq \rho(\psi_{\triangle PQR})\geq\rho(\psi_{\triangle PQR_1}).
\end{align*}
From Lemma \ref{pentagon-2-1} and \ref{pentagon-5}, we see $\rho(\psi_{\triangle PQR_i})=\frac{2}{5}$ for $i=1,2$.
Therefore, we have $\rho(\psi_{\triangle PQR})=\frac{2}{5}$.
\end{proof}

Lemma \ref{pentagon-5} can be extended to the following lemma.

\begin{lem}\label{pentagon-7}
Let $\triangle PQR$ be a triangle in $D$ with $\delta(P,Q,R)=\Delta_2'(P,Q)$.
Then, $\rho(\psi_{\triangle PQR})=\dfrac{2}{5}$ holds.
\end{lem}
\begin{proof}
Let $S$ be a point on the line $PQ$ such that line $RS$ is perpendicular to the line $PQ$ in the hyperbolic sense.
If $S$ is in the line segment $PQ$, from Lemma \ref{pentagon-5}, we see that $\rho(\psi_{\triangle PQR})=\dfrac{2}{5}$.
We suppose that $S$ is not in the line segment $PQ$.
 We also assume that $S$ is closer to $P$ than $Q$ without loss of generality (see Figure \ref{fig:pentagon5}).
We set $h=\delta(R,Q,P)$ and $\theta=\angle RQS$.
We will show
\begin{align}\label{Delta2RQ}
\Delta_2'(R,Q)\leq h\leq \dfrac{1}{2}\Delta_1'(R,Q).
\end{align}
Through the law of right-angled  triangles \cite{B}, 
we have
\begin{align*}
&\sinh d'(Q,R)\sin \theta=\sinh d'(S,R),\\
&\sinh d'(Q,P)\sin \theta=\sinh h,
\end{align*}
which implies 
\begin{align}\label{sinhh}
\sinh h=\dfrac{\sinh d'(S,R)\sinh d'(Q,P)}{\sinh d'(Q,R)}.
\end{align}
Now, we see
\begin{align}\label{sinhDelta2}
\sinh \Delta_2'(R,Q)=\dfrac{1}{2}\left(\dfrac{e^{2 d'(R,Q)}+1}{e^{2 d'(R,Q)}-1}-\dfrac{e^{2 d'(R,Q)}-1}{e^{2 d'(R,Q)}+1}\right)=\dfrac{2e^{2 d'(R,Q)}}{e^{4 d'(R,Q)}-1}.
\end{align}
Since $h=\Delta_2'(P,Q)$ holds, we have 
\begin{align}\label{sinhhdRQ}
\sinh h=\dfrac{2e^{2 d'(P,Q)}}{e^{4 d'(P,Q)}-1}.
\end{align}

From (\ref{sinhh}),  (\ref{sinhDelta2}) and (\ref{sinhhdRQ}), the inequality $\sinh\Delta_2'(R,Q)\leq \sinh h$ is 
equivalent to 
\begin{align*}
\dfrac{2e^{2 d'(R,Q)}\sinh d'(Q,R)}{e^{4 d'(R,Q)}-1}\leq \dfrac{2e^{ d'(P,Q)}\sinh d'(Q,P)}{e^{4 d'(P,Q)}-1},	
\end{align*}
which is 
equivalent to 
\begin{align*}
\cosh d'(P,Q)\leq \cosh d'(R,Q),
\end{align*}
which is apparent.
Therefore, we have $\Delta_2'(R,Q)\leq h$.
From the fact that $\sinh \Delta_1'(R,Q)=\dfrac{2e^{ d'(R,Q)}}{e^{2 d'(R,Q)}-1}$, 
the inequality $h\leq \dfrac{1}{2}\Delta_1'(R,Q)$ 
is equivalent to 
\begin{align}\label{2sinhh}
2\sinh h\cosh h\leq \dfrac{2e^{ d'(R,Q)}}{e^{2 d'(R,Q)}-1}.
\end{align}
Applying the equations  (\ref{sinhh}) and (\ref{sinhhdRQ}) to the inequality (\ref{2sinhh}),
the inequality (\ref{2sinhh}) is equivalent to
\begin{align*}
\dfrac{2\sinh d'(P,Q)e^{2d'(P,Q)}\cosh h}{e^{4d'(P,Q)}-1}\leq \dfrac{e^{d'(R,Q)}\sinh d'(R,Q)}{e^{2d'(R,Q)}-1},
\end{align*}
which reduces to
\begin{align*}
\cosh h\leq \dfrac{1}{2}\dfrac{e^{2d'(P,Q)}+1}{e^{d'(P,Q)}},
\end{align*}
which is 
equivalent to 
\begin{align}\label{sinhhsinhPQ}
\sinh h\leq \sinh d'(P,Q).
\end{align}
Considering (\ref{sinhh}), the inequality (\ref{sinhhsinhPQ}) is equivalent to
\begin{align*}
\sinh d'(S,R)\leq \sinh d'(Q,R),
\end{align*}
which holds by the law of right-angled  triangles.
Therefore, the inequality (\ref{Delta2RQ}) holds.
Considering $\sinh d'(Q,R)\leq \sinh d'(Q,P)$, from Lemma \ref{pentagon-6} we have $\rho(\psi_{\triangle PQR})=\dfrac{2}{5}$.
\end{proof}

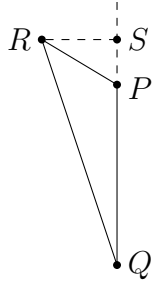
\begin{figure}
    \centering
\begin{tikzpicture}
\draw (0,1.2)--(0,-1.2)--(-1,1.8)--cycle;

\draw (-1,1.8)node[left]{$R$};
\draw (0.0,1.15)node[right]{$P$};
\draw (0,-1.2)node[right]{$Q$};
\draw (0,1.8)node[right]{$S$};

\coordinate (P) at (0,1.2);
\fill (P) circle [radius=1.5pt];
\coordinate (Q) at (0,-1.2);
\fill (Q) circle [radius=1.5pt];
\coordinate (R) at (-1,1.8);
\fill (R) circle [radius=1.5pt];
\coordinate (S) at (0,1.8);
\fill (S) circle [radius=1.5pt];

\draw[dashed] (0,2.3)--(0,1.2);
\draw[dashed] (0,1.8)--(-1,1.8);

 \end{tikzpicture}
    \caption{$\triangle PQR$ and $S$ in Lemma \ref{pentagon-7} }
    \label{fig:pentagon5}
\end{figure}

The following theorem is one of the main results of the present paper.

\begin{thm}\label{pentagon-8}
Let $\triangle PQR$ be a triangle in $D$ such that
$\Delta_2'(P,Q)\leq \delta(P,Q,R)\leq \frac{1}{2}\Delta_1'(P,Q)$
or $\frac{1}{2}\Delta_1'(P,Q)\leq \delta(P,Q,R)\leq \Delta_2'(P,Q)$.
Then, $\rho(\psi_{\triangle PQR})=\frac{2}{5}$ holds.
\end{thm}
\begin{proof}
We take the points $R_1$ and $R_2$ as in the proof of Lemma \ref{pentagon-6}.
Then, we have 
\begin{align*}
&\rho(\psi_{\triangle PQR_2})\geq \rho(\psi_{\triangle PQR})\geq\rho(\psi_{\triangle PQR_1}),\\
&\text{or}\\
&\rho(\psi_{\triangle PQR_2})\leq \rho(\psi_{\triangle PQR})\leq\rho(\psi_{\triangle PQR_1}).
\end{align*}
By Lemmas \ref{pentagon-2-1} and  \ref{pentagon-7} we have $\rho(\psi_{\triangle PQR})=\frac{2}{5}$.

\end{proof}

\section{Additional results}

The following lemma is an extension of Theorem 3.1(\cite{NY}).
 
\begin{lem}\label{further-1}
Let $P_1, P_2$ be two distinct points in $D$.
Let $v_1, v_2$ be intersections of the line $P_1P_2$ and $S^1$ such that  $v_1$ is closer to $P_1$ than $P_2$.
Let $n\in {\mathbb{Z}_{>0}}$.
For $P\in D$ such that $P$ is not on the line $P_1P_2$,
we define  $\tau_n(P)$ as
\begin{align*}
\sharp \{w\in S^1 \mid w\ne v_i(i=1,2), P\in\ \text{the line }w\psi_{P_1P_2}^{2n}(w)\}.  
\end{align*}
Then, 
\begin{align*}
\tau_n(P)=
\begin{cases}
0&\text{ if }\delta(P_1, P_2,P)<\Delta_{n}'(P_1, P_2),\\
1&\text{ if }\delta(P_1, P_2,P)=\Delta_{n}'(P_1, P_2),\\
2&\text{ if }\delta(P_1, P_2,P)>\Delta_{n}'(P_1, P_2).
\end{cases}
\end{align*}
\end{lem}
\begin{proof}
There exists  an isometry $f:D\to {\mathbb H}$ such that
 $f(P_1)=ki$ and $f(P_2)=i$, where $k>1$ and $f$. The map $f$ is naturally extended to the map $\Bar{f}$ from $D\cup S^1$ to 
 ${\mathbb H}\cup {\mathbb R} \cup \{\infty\}$ such that $\Bar{f}(v_1)=\infty$ and $\Bar{f}(v_2)=0$. 
Let $\phi:=f\circ\psi_{P_1P_2}\circ f^{-1}$.
Let $u+vi=f(P)$ with $u,v\in \mathbb{R}$.
We suppose that $u>0$.
In  hyperbolic geometry, a line in ${\mathbb H}$ is a semicircle that intersects the real axis at a right angle.
We consider a semicircle $C$ in ${\mathbb H}$ with center $\alpha$, where $-\beta$ and $\gamma i$ lie on $C$, for $\alpha,\beta, \gamma\in {\mathbb R}$.
Then, we have
\begin{align*}
\alpha^2+\gamma^2=(\alpha+\beta)^2,
\end{align*}
which implies $\alpha=\frac{\gamma^2-\beta^2}{2\beta}$.
Thus, $2\alpha+\beta=\frac{\gamma^2}{\beta}$ is on $C$.
Therefore, for $w\in {\mathbb R}$ with $w\ne 0$, we have
\begin{align}
\phi(w)=\begin{cases}
-\frac{k^2}{w}&\text{if\ }w>0,\\
-\frac{1}{w}&\text{if\ }w<0.
\end{cases}
\end{align} 
Therefore, we have for $w\in \mathbb{R}_{>0}$ and  $n\in {\mathbb{Z}_{>0}}$,
\begin{align*}
\phi^{2n}(w)=\frac{w}{k^{2n}}.
\end{align*}
Let $C'$ be the circle whose center is on the real axis and which passes through  $w$ and 
$\frac{w}{k^{2n}}$.
The following is the condition $u+vi\in C'$:
\begin{align*}
\left(u-\dfrac{1+k^{2n}}{2k^{2n}}w\right)^2+v^2=\left(\dfrac{k^{2n}-1}{2k^{2n}}w\right)^2,
\end{align*}
which implies 
\begin{align}\label{(u^2+v^2)}
w^2-(k^{2n}+1)uw+k^{2n}(u^2+v^2)=0
\end{align}
The discriminant of formula (\ref{(u^2+v^2)}) for w is 
\begin{align*}
\mathrm{D}=(k^{2n}+1)^2u^2-4k^{2n}(u^2+v^2).
\end{align*}
Since $\mathrm{D}=((k^{2n}-1)u-2k^{n}v)((k^{2n}-1)u+2k^{n}v)$ and $u,v,k-1>0$, we have
\begin{align}\label{m=begincases}
\tau_n(P)=\begin{cases}0 & \text{if }v>\left(\dfrac{k^{2n}-1}{2k^{n}}\right)u,\\
1 & \text{if }v=\left(\dfrac{k^{2n}-1}{2k^{n}}\right)u,\\
2 & \text{if }v<\left(\dfrac{k^{2n}-1}{2k^{n}}\right)u.
\end{cases}
\end{align}
We denote the set $\{x+yi\in {\mathbb H}\mid y=\left(\dfrac{k^{2n}-1}{2k^{n}}\right)x\}$ by $L$.
Let $\theta\in (0,\pi/2)$ be the angle between the imaginary axis and $L$.
We will calculate the length $d$ of a vertical line from each point of $L$ to the imaginary axis in the hyperbolic sense. Let $x+iy\in L$ with $r=\sqrt{x^2+y^2}$.
Then, the length of the curve $r\sin w +ir\cos w (0\leq w\leq \theta)$ is
\begin{align}\label{int_0}
d=\int_0^{\theta} \dfrac{dw}{\cos w}= \log\left(\dfrac{1+\tan(\theta/2)}{1-\tan(\theta/2)}\right).
\end{align}
We remark that $d$ depends only on $\theta$.
On the other hand, from the fact that $\tan \theta=\dfrac{2k^{n}}{k^{2n}-1}$ we see that $\theta=\pi-2\arctan k^n$.
Therefore, from (\ref{int_0}),
we have
\begin{align}\label{d=log}
d=\log\left(\dfrac{k^{n}+1}{k^{n}-1}\right).
\end{align}
The distance between $i$ and $ki$ is
\begin{align}\label{int_1}
\int_1^{k} \dfrac{dy}{y}=\log k.
\end{align}
From (\ref{d=log}) and (\ref{int_1}), we have
\begin{align}\label{d=Delta}
d=\Delta_n'(P_1,P_2).
\end{align}
Similarly, we can consider the case of $u<0$.
The theorem follows
 from (\ref{m=begincases}) and (\ref{d=Delta}).
\end{proof}

We provided the number of periodic points for the case of $\rho=1/3$ in \cite{NY}.
For the case of $\rho=2/5$, we present an upper bound for the number of periodic points.

\begin{thm}\label{further-2} 
Let $\triangle P_1P_2P_3$ be a triangle in $D$ with $\rho(\psi_{\triangle P_1P_2P_3})=\frac{2}{5}$.
Then, the number of periodic orbits of $\psi_{\triangle P_1P_2P_3}$ with period $5$ is less than or equal to $6$.
\end{thm}
\begin{proof}
We assume that $\triangle P_1P_2P_3$ has vertices $P_1, P_2$, and $P_3$ listed in counterclockwise order without loss of generality. Let $v\in {S^1}$ be a periodic point of $\psi_{\triangle P_1P_2P_3}$ with  period $5$.
We set $v_i=\psi_{\triangle P_1P_2P_3}^i(v)$ for  $i\in \mathbb{Z}_{\geq 0}$.
It is easy to see that if there exists a point $P_j$ for $1\leq j\leq 3$ such that $P_j$ lies on the line $v_iv_{i+1}$ for three different values of $i$ with $0\leq i \leq 4$, then a contradiction arises.
First, we suppose that for each line $v_iv_{i+1}$ for $0\leq i \leq 4$ only one vertex of $\triangle P_1P_2P_3$ is on its line.  
If there exists a point $P_j$ that does not lie on any of the lines $v_iv_{i+1}$ with $0\leq i \leq 4$, then there must exist a point $P_k$ that lies on at least three of these lines. This leads to a contradiction.
Therefore, there exists a point $P_l$ which is only on one line among the lines  $v_iv_{i+1}$ for $0\leq i \leq 4$. 
Thus, there must exist a point $P_l$ that lies on only one of the lines $v_iv_{i+1}$ with $0\leq i \leq 4$.
Without loss of generality, we assume that $P_1$ lies on the line $v_0v_{1}$ but not on any of the lines $v_iv_{i+1}$ for $1\leq i \leq 4$.
Then, we see that $\psi_{P_2P_3}^i(v_1)=v_{i+1}$ $1\leq i \leq 4$ and $P_1$ is on the line $v_5v_{1}$.
We now assume that there exists an index $i$ with $0\leq i \leq 4$ such that two vertices of $\triangle P_1P_2P_3$ lie on the line $v_iv_{i+1}$.
We suppose that $P_1$ and $P_2$ are  on the line $v_0v_{1}$ 
without loss of generality.
If $P_1$ does not lie on any of the lines $v_iv_{i+1}$ with $1\leq i \leq 4$, then either $P_2$ or $P_3$ lies on the line $v_iv_{i+1}$ for three different values of $i$ with $0\leq i\leq 4$, leading to a contradiction.
Thus, there exists a unique index $j$ with $1\leq j\leq 4$ such that $P_1$ lies on the line $v_jv_{j+1}$.
We observe that $\psi_{P_2P_3}^i(v_{j+1})=v_{j+1+i}$ for $1\leq i\leq 4$, and since $P_1$ lies on the line $v_jv_{j+1}$, it must lie on the line $v_{j+5}v_{j+6}$.
From Lemma \ref{further-1}, we have
\begin{align*}
\sharp \{v\in S^1\mid \text{$P_1$ is on the line $v\psi_{P_2P_3}^4(v)$}\}\leq 2.
\end{align*}
Therefore, the number of periodic orbits of $\psi_{\triangle P_1P_2P_3}$ with period $5$ is less than or equal to $6$.

\end{proof}

In Theorem \ref{pentagon-8} we give a sufficient condition for $\rho=2/5$.
However, the following theorem shows that if the condition of Theorem \ref{pentagon-8} is not satisfied in a given case, then $\rho=2/5$ does not hold in that case.
Thus, Theorem \ref{pentagon-8} provides a weak sufficient condition for $\rho=2/5$.
\begin{thm}\label{further-3}
Let $\triangle P_1P_2P_3$ be a triangle in $D$.
If for all $i,j,k$ with $\{i,j,k\}=\{1,2,3\}$,
$\Delta_2'(P_i,P_j)>\delta(P_i,P_j,P_k)$, then we have $\rho(\psi_{\triangle P_1P_2P_3})>\frac{2}{5}$.
\end{thm}
\begin{proof}
Let $\triangle P_1P_2P_3$ be a triangle in $D$ with the the given hypothesis.
We suppose that $\rho(\psi_{\triangle P_1P_2P_3})=\frac{2}{5}$. 
From the proof of Theorem \ref{further-2}, we know that there exist $v \in D$ and indices $i,j,k$ with ${i,j,k}={1,2,3}$ such that $\psi_{P_iP_j}^m(v) = \psi_{\triangle P_1P_2P_3}^m(v)$ for $1\leq m\leq 4$, and $P_k$ lies on the line $v\psi_{P_iP_j}^4(v)$.
Lemma \ref{further-1} implies that $\Delta_2'(P_i,P_j) \leq \delta(P_i,P_j,P_k)$. However, this contradicts the hypothesis of the theorem.
Thus, we have shown that $\rho(\psi_{\triangle P_1P_2P_3})\ne\frac{2}{5}$. Moreover, we can easily find a point $P\in D$ such that $\triangle P_1P_2P_3 \subset \triangle P_iP_jP$ and $\delta(P_i,P_j,P) = \Delta_2'(P_i,P_j)$. By Theorem \ref{pentagon-8}, we know that $\rho(\psi_{\triangle P_iP_jP})=\frac{2}{5}$. Therefore, we conclude that $\rho(\psi_{\triangle P_1P_2P_3})>\frac{2}{5}$.\end{proof}

\section{Isosceles triangle}
In this chapter, we consider a condition that $\rho=2/5$ for sufficiently large isosceles triangles.

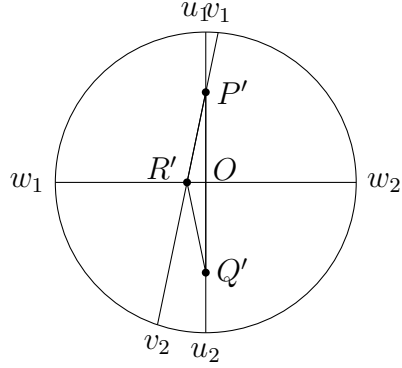
\begin{figure}
    \centering
\begin{tikzpicture}
\draw(0,0) circle (2);
\draw (0,1.2)--(0,-1.2)--(-0.2490,0)--cycle;
\draw (0.164609987,1.9932143)--(-0.6421009,-1.894124155);
\draw (-2,0)--(2,0);
\draw (0,-2)--(0,2);

\draw (-0.23,0.2)node[left]{$R'$};
\draw (-0.05,0.17)node[right]{$O$};
\draw (0,1.2)node[right]{$P'$};
\draw (0,-1.2)node[right]{$Q'$};
\draw (0.164609987,1.9932143)node[above]{$v_1$};
\draw (-0.6421009,-1.894124155)node[below]{$v_2$};
\draw (-0.13,2.0)node[above]{$u_1$};
\draw (0,-2.0)node[below]{$u_2$};
\draw (-2,0)node[left]{$w_1$};
\draw (2,0)node[right]{$w_2$};

\coordinate (P) at (0,1.2);
\fill (P) circle [radius=1.5pt];
\coordinate (Q) at (0,-1.2);
\fill (Q) circle [radius=1.5pt];
\coordinate (R) at (-0.2490,0);
\fill (R) circle [radius=1.5pt];

 \end{tikzpicture}
    \caption{Isosceles triangle in Lemma \ref{Isosceles-1}. }
    \label{fig:pentagon7}
\end{figure}

\begin{thm}\label{Isosceles-1}
Let $\triangle PQR$ be an isosceles triangle  in $D$ with $d'(P,R)=d'(Q,R)$, $d'(P,Q)>\log 3$ and
$\delta(P,Q,R)<\Delta_2'(P,Q)$.
Then, $\rho(\psi_{\triangle PQR})>\frac{2}{5}$.
\end{thm}
\begin{proof}
Let
\begin{align}\label{t=dfraced}
t=\dfrac{e^{d'(P,Q)}-1}{e^{d'(P,Q)}+1},
\end{align}
which implies $1/2<t<1$.
We put $P'=(0,t)$, and $Q'=(0,-t)$.
Then, since $d'(P,Q)=d'(P',Q')$ holds, 
there exist an isometry $\phi$ on $D$
such that $\phi(P)=P'$ and $\phi(Q)=Q'$. 
We put $R'(r,0)=\phi(R)$, $\theta=\angle R'P'Q'$, and
$h=\delta(P',R',Q')$.
We note that $\delta(P,Q,R)=\delta(P',Q',R')=d'(O,R')$, where 
$O=(0,0)$.
We suppose that $r<0$ without loss of generality.
By the law of right-angled  triangles \cite{B}, 
we have
\begin{align*}
&\sinh d'(Q',P')\sin \theta=\sinh h,\\
&\sinh d'(P',R')\sin \theta=\sinh d'(O,R').
\end{align*}
By eliminating $\sin \theta$ from above equations, we have
\begin{align}\label{sinhhdfrac}
\sinh h=\dfrac{\sinh d'(O,R') \sinh d'(P',Q')}{\sinh d'(P',R')}.
\end{align}
From the fact that $\sinh \Delta_2'(P',Q')=\dfrac{2e^{2d'(P',Q')}}{e^{4d'(P',Q')}-1}$, the hypothesis of the theorem,  and 
(\ref{sinhhdfrac}), we have
\begin{align}\label{sinhh<}
\sinh h<\dfrac{2e^{2d'(P',Q')}\sinh d'(P',Q')}{(e^{4d'(P',Q')}-1)\sinh d'(P',R')}.
\end{align}
From (\ref{t=dfraced}), we have 
\begin{align}\label{2e^2d}
\dfrac{2e^{2d'(P',Q')}\sinh d'(P',Q')}{(e^{4d'(P',Q')}-1)}=\dfrac{1-t^2}{2(1+t^2)}.
\end{align}
We will show 
\begin{align}\label{1-t^22}
\dfrac{1-t^2}{2(1+t^2)}<\dfrac{e^{d'(P',R')}}{e^{2d'(P',R')}+1},
\end{align}
which yields that by considering (\ref{sinhh<}) and (\ref{2e^2d})
\begin{align}\label{h<Delta2}
h<\Delta_2'(P',R').
\end{align}
Let $v_1,v_2\in {S^1}$ be points at infinity, which are intersection points with the line $P'R'$ and ${S^1}$.
Let $v_1$ be closer to $P'$ than $Q'$.
We put $u_1=(0,1),u_2=(0,-1), w_1=(-1,0)$, and $w_2=(0,1)$. See Figure \ref{fig:pentagon7}.
We have
\begin{align}\label{dPR=}
d'(P',R')&=\dfrac{1}{2}\log\dfrac{|v_2P'||v_1R'|}{|v_1P'||v_2R'|}\nonumber\\
&<\dfrac{1}{2}\log\dfrac{|u_2P'||w_2R'|}{|u_1P'||w_1R'|}\nonumber\\
&=\dfrac{1}{2}\log\dfrac{(1+t)(1+|r|)}{(1-t)(1-|r|)}.
\end{align}
From the fact that $d'(O,R')<\Delta_2'(P,Q)$ and Lemma \ref{pentagon-3}, we have
\begin{align*}
|r|<\dfrac{(t + 1)^2(1 - t)^2}{t^4 + 6t^2 + 1},
\end{align*}
which yields 
\begin{align*}
\dfrac{1+|r|}{1-|r|}<\dfrac{(t^2+1)^2}{4t^2}\leq \dfrac{1+t}{1-t},
\end{align*}
by $1/2<t<1$.
Therefore, we have
\begin{align}\label{1t1t}
d'(P',R')<\log\dfrac{1+t}{1-t}.
\end{align}
On the other hand, according to the hypothesis of the theorem, we have
\begin{align}\label{log3}
d'(P',R')>d'(P',Q')>\dfrac{1}{2}\log 3.
\end{align}
Using (\ref{1t1t}) and (\ref{log3}), we can deduce the inequality (\ref{1-t^22}). Therefore, we have the inequality (\ref{h<Delta2}). Finally, the required conclusion follows from Theorem \ref{further-3}.
\end{proof}

Next, we will consider the upper bound on the height of an isosceles triangle with $\rho = \frac{2}{5}$.
 For this, we require a triangle $\triangle PQR$ as considered in Lemma \ref{pentagon-2}. Let $0 < t < 1$,
 and define $P=(0,t)$, $Q=(0,-t)$, and $R=\left(\frac{t-1}{t+1},0\right)$. Note that $\delta(P,Q,R)=\frac{1}{2}\Delta_1'(P,Q)$.
Furthermore, we define $A_1=(1,0)$, $A_2=\left(\frac{t^2-1}{t^2+1},\frac{2t}{t^2+1}\right)$, $A_3=(0,-1)$,
 $A_4=(0,1)$, and $A_5=\left(\frac{t^2-1}{t^2+1},-\frac{2t}{t^2+1}\right)$. It can be verified that $\psi_{\triangle PQR}(A_i)=A_{i+1}$ for $1\leq i\leq 4$, and
$\psi_{\triangle PQR}(A_5)=A_{1}$.
Let $u_2,u_3\in {S^1}$ be points at infinity, which are intersection points with the line $QR$ and ${S^1}$
as shown in Figure \ref{fig:pentagon8}.
We put $B_1=(-1,0)$.
Let $u_1\in {S^1}$ be a point at infinity, which is  an intersection point with the line $u_2P$ and ${S^1}$
as shown in Figure \ref{fig:pentagon8}.
We note that  $\psi_{\triangle PQR}(u_i)=u_{i+1}$ for $i=1,2$.
We put $u_{3+i}=\psi_{\triangle PQR}^i(u_3)$ for $1\leq i\leq 3$.
We will use this notation going forward.

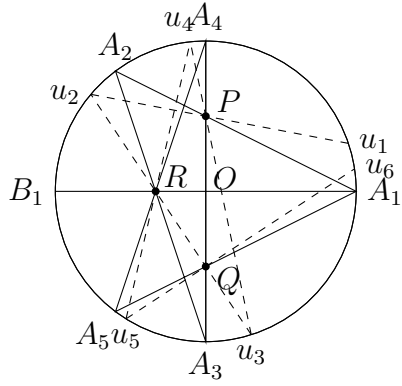
\begin{figure}
    \centering
\begin{tikzpicture}
\draw(0,0) circle (2);

\draw(0,0) circle (2);
\draw (2,0)--(-1.2,1.6)--(0,-2)--(0,2)--(-1.2,-1.6)--cycle;
\draw (-2,0)--(2,0);
\draw (0,-2)--(0,2);
\draw[dashed] (1.89525,0.6387)--(-1.5274,1.291);
\draw[dashed] (0.60433,-1.9065)--(-1.5274,1.291);
\draw[dashed] (0.60433,-1.9065)--(-0.2057203,1.9893916);
\draw[dashed]  (-1.059678,-1.6961959)--(-0.2057203,1.9893916);
\draw[dashed]  (-1.059678,-1.6961959)--(1.9774947,0.2991895871);

\draw (-0.1,0.2)node[left]{$R$};
\draw (-0.05,0.17)node[right]{$O$};
\draw (0,1.2)node[right]{$P$};
\draw (0,-1.2)node[right]{$Q$};
\draw (-1.5,-1.6)node[below]{$A_5$};
\draw (0,2.0)node[above]{$A_4$};
\draw (0,-2.0)node[below]{$A_3$};
\draw (-1.2,1.6)node[above]{$A_2$};
\draw (2,0)node[right]{$A_1$};
\draw (-2,0)node[left]{$B_1$};

\draw (1.89525,0.6387)node[right]{$u_1$};
\draw (-1.5274,1.291)node[left]{$u_2$};
\draw (0.60433,-1.9065)node[below]{$u_3$};
\draw (-0.3557203,1.9893916)node[above]{$u_4$};
\draw (-1.059678,-1.6961959)node[below]{$u_5$};
\draw (1.9774947,0.2991895871)node[right]{$u_6$};

\coordinate (P) at (0,1);
\fill (P) circle [radius=1.5pt];
\coordinate (Q) at (0,-1);
\fill (Q) circle [radius=1.5pt];
\coordinate (R) at (-0.666,0);
\fill (R) circle [radius=1.5pt];

 \end{tikzpicture}
    \caption{$A_i$$(1\leq i \leq 5)$, $u_i$$(1\leq i \leq 6)$, and $B_1$. }
    \label{fig:pentagon8}
\end{figure}

\begin{lem}\label{Isosceles-2}
Let $0.8<t<1$.
Then, $\dfrac{|Pu_2|}{|Pu_1|}<\dfrac{1}{3}$ holds.
\end{lem}
\begin{proof}
Let $u_1=(\alpha_1,\beta_1)$.
Then, we have
\begin{align*}
\beta_1=\dfrac{3t^5+10t^3+3t-(t^4+3t^3+3t^2+t)\sqrt{4t^3+t^2-2t+1}}{t^6+4t^5+15t^4+4t^3+7t^2+1}.
\end{align*}
We show that $0<\beta_1<1/3$.
From Figure \ref{fig:pentagon8}, it is evident that $\beta_1$ is positive. We can also observe that $f_1(t):= 4t^3 +t^2-2t+1$ is monotonically increasing for $0.8 < t < 1$ and $f_1(0.8) > 2$. Thus, we have $f_1(t)>2$.

Therefore, we have
\begin{align*}
\beta_1<\dfrac{3t^5+10t^3+3t-(t^4+3t^3+3t^2+t)}{t^6+4t^5+15t^4+4t^3+7t^2+1}.
\end{align*}
We see 
\begin{align*}
\dfrac{3t^5+10t^3+3t-(t^4+3t^3+3t^2+t)}{t^6+4t^5+15t^4+4t^3+7t^2+1}<\dfrac{1}{3},
\end{align*}
is equivalent to
\begin{align*}
S:=t(\dfrac{16}{3}t-3)+\dfrac{1}{3}(t(t-1)(t^4-4t^3+14t^2-3t-3)+1)>0.
\end{align*}
It is easily seen that $t(\frac{16}{3}t-3)>0$.
We set $f_2(t):=t^4-4t^3+14t^2-3t-3$.
Since $f_2(t)$ is monotonically increasing for $0.8<t<1$,
we have $1.9216\leq f_2(t)\leq 5$.
Next, we consider the term $\frac{1}{3}(t(t-1)f_2(t)+1)$. Since $f_2(t)$ is bounded by $1.9216 \leq f_2(t) \leq 5$ for $0.8 < t < 1$, we can use the fact that $|t(t-1)| \leq 0.8 \cdot 0.2 = 0.16$ on the interval to obtain the inequality $|t(t-1)f_2(t)| \leq 0.16 \cdot 5 < 1$. Therefore, we have $\frac{1}{3}(t(t-1)f_2(t)+1) > 0$.
Combining these observations, we conclude that $S > 0$, as desired.

We put $u=(2\sqrt{2}/3,1/3)\in S^1$ and $u'(\ne u)$ is the intersection point with the line $uP$ and ${S^1}$.
Then, we have
\begin{align*}
u'=\left(\dfrac{2\sqrt{2}(t^2-1)}{3t^2-2t+3},-\dfrac{t^2-6t+1}{3t^2-2t+3} \right).
\end{align*}
We have
\begin{align*}
\dfrac{|Pu'|}{|Pu|}=\dfrac{3(1-t^2)}{3t^2-2t+3}<\dfrac{3(1-t^2)}{3.32}<\dfrac{1}{3}.
\end{align*}
From Lemma \ref{fund-1} and Lemma \ref{fund-1-2}, we see
\begin{align*}
\dfrac{|Pu_2|}{|Pu_1|}\leq\dfrac{|Pu'|}{|Pu|}.
\end{align*}
Thus, we have the desired conclusion.
\end{proof}

\begin{lem}\label{Isosceles-3}
Let $0.8<t<1$.
Then, $\dfrac{|Ru_3|}{|Ru_2|}<\dfrac{1}{t}$ holds.
\end{lem}
\begin{proof}
We have
\begin{align*}
\dfrac{|Ru_3|}{|Ru_2|}<\dfrac{|RA_1|}{|RB_1|}=\dfrac{1}{t}.
\end{align*}
\end{proof}

\begin{lem}\label{Isosceles-4}
Let $0.8<t<1$.
Then, $\dfrac{|Pu_4|}{|Pu_3|}\leq \dfrac{7}{10}(1-t)$
 holds.
\end{lem}
\begin{proof}
Let $u_3=(\alpha_3,\beta_3)$.
Then, we have
\begin{align*}
\alpha_3=\dfrac{t^4-t^2+(1-t)\sqrt{4t^3+t^2-2t+1}}{t^4+2t^3+2t^2-2t+1}.
\end{align*}
We see that $4t^3+t^2-2t+1< t^4+2t^3+2t^2-2t+1$ and 
$4t^3+t^2-2t+1>1$, which implies 
\begin{align*}
\alpha_3<\dfrac{(1-t)}{\sqrt{4t^3+t^2-2t+1}}\leq 1-t.
\end{align*}
We put $V=(1-t,0)$.
Let $V_1,V_2\in {S^1}$ be points at infinity, which are intersection points with the line $VP$ and ${S^1}$
such that $|V_1P|<|V_1V|$.
Then, we have
\begin{align*}
\dfrac{|Pu_4|}{|Pu_3|}< \dfrac{|PV_2|}{|PV_1|}=\dfrac{\sqrt{-t^4+2t^3+t^2-2t+1}-t^2}{\sqrt{-t^4+2t^3+t^2-2t+1}+t^2}.
\end{align*}
Using $\sqrt{-t^4+2t^3+t^2-2t+1}+t^2\geq \frac{10}{7}$ and $\sqrt{-t^4+2t^3+t^2-2t+1}-t^2\leq 1-t$,
we arrive at the desired conclusion.
\end{proof}

\begin{lem}\label{Isosceles-5}
Let $0.8<t<1$.
Then, we have:
\begin{enumerate}
\item[(1)]
$\dfrac{|Ru_5|}{|Ru_4|}< \dfrac{1}{t}$.
\item[(2)]
$\dfrac{|Qu_6|}{|Qu_5|}< \dfrac{1+t}{1-t}$.
\end{enumerate}

\end{lem}
\begin{proof}
For (1) we see
\begin{align*}
\dfrac{|Ru_5|}{|Ru_4|}<\dfrac{|RA_1|}{|RB_1|}=\dfrac{1}{t}.
\end{align*}
For (2) we see
\begin{align*}
\dfrac{|Qu_6|}{|Qu_5|}<\dfrac{|QA_4|}{|QA_3|}=\dfrac{1+t}{1-t}.
\end{align*}

\end{proof}

\begin{lem}\label{Isosceles-6}
Let $0.8<t<1$.
Then, $\psi'_P(u_1)\psi'_R(u_2)\psi'_P(u_3)\psi'_R(u_4)\psi'_Q(u_5)<1$
 holds.
\end{lem}
\begin{proof}
Based on Lemmas \ref{fund-1}, \ref{Isosceles-2}, \ref{Isosceles-3}, \ref{Isosceles-4}, \ref{Isosceles-5} and \ref{Isosceles-6},
 we conclude that  
\begin{align*}
\psi'_P(u_1)\psi'_R(u_2)\psi'_P(u_3)\psi'_R(u_4)\psi'_Q(u_5)<\dfrac{7(t+1)}{30t^2}.
\end{align*}
For $0.8<t<1$, it is easy to see that $\frac{7(t+1)}{30t^2}<1$.
\end{proof}

\begin{lem}\label{Isosceles-7}
Let $0.8<t<1$.
Let $v_1\in arc(A_1,u_1)$ and $v_{i+1}=\psi_{\triangle PQR}^{i}(v_1)$ for $1\leq i \leq 4$.
Then, $\psi'_P(v_1)\psi'_R(v_2)\psi'_P(v_3)\psi'_R(v_4)\psi'_Q(v_5)<1$
 holds.
\end{lem}
\begin{proof}
We observe that $v_i$ lies on the  $arc(A_i, u_i)$ for $1 \leq i \leq 5$. Using Lemma \ref{fund-1-2}, we can see that for $1 \leq i \leq 5$, $\psi'{\triangle PQR}(u_i) > \psi'{\triangle PQR}(v_i)$. This inequality, together with Lemma \ref{Isosceles-6}, implies Lemma \ref{Isosceles-7}.
\end{proof}

\begin{lem}\label{Isosceles-8}
Let $0.8<t<1$.
Then, $u_6\in arc(A_1,u_1)$. 
\end{lem}
\begin{proof}
It is not difficult to see that $u_6\in arc(A_1,A_4)$.
We suppose that $u_6\in arc[u_1,A_4)$.
Since for $v\in arc(A_1,u_1)$ $\psi_{\triangle PQR}^5(v)=\psi_Q\circ\psi_R\circ\psi_P\circ\psi_R\circ \psi_P(v)$
,  $\psi_{\triangle PQR}^5$  is differentiable on $arc(A_1,u_1)$.
As $\psi_{\triangle PQR}^5(A_1) = A_1$ and $\psi_{\triangle PQR}^5(u_1) \in arc[u_1, A_4)$, by the mean value theorem,
 there exists some $v \in arc(A_1, u_1)$ such that $(\psi_{\triangle PQR}^5)'(v) \geq 1$. However, 
 this contradicts Lemma \ref{Isosceles-7}. Hence, we can conclude that $u_6 \in arc(A_1, u_1)$.
\end{proof}

Suppose there exists $v\in S^1$ such that $\psi_{\triangle PQR}^5(v) = v$. 
Let $\mathcal{P}(v)$ denote the pentagram formed by the edges 
$\psi_{\triangle PQR}^{i-1}(v)\psi_{\triangle PQR}^{i}(v)$ for $1\leq i \leq 5$.
Let $C\in D$. We use $E(\mathcal{P}(v),C)$ to denote the number of edges of $\mathcal{P}(v)$ that include $C$.
We note that since two consecutive edges of $\mathcal{P}(v)$ do not share a vertex of $\triangle PQR$,
we have $E(\mathcal{P}(v),C)\leq 2$ for all $C\in \{P,Q,R\}$.

\begin{lem}\label{Isosceles-9}
Let $0.8<t<1$.
Then, there does not exist $v\in S^1$  with  $\psi_{\triangle PQR}^5(v)=v$
such that $E(\mathcal{P}(v),P)=2, E(\mathcal{P}(v),Q)=1$, and
$E(\mathcal{P}(v),R)=2$. 
\end{lem}
\begin{proof}
We suppose that there exists $v\in S^1$  with  $\psi_{\triangle PQR}^5(v)=v$
such that $E(\mathcal{P}(v),P)=2, E(\mathcal{P}(v),Q)=1$, and
$E(\mathcal{P}(v),R)=2$.
We put $v_i=\psi_{\triangle PQR}^{i-1}(v)$ for $1\leq i \leq 5$.
We also suppose that $Q$ is on the edge $v_5v_1$ without loss of generality.
Since $E(\mathcal{P}(A_1),Q)=2$ holds, $v_1\ne A_i$ for all $i$.
Therefore, $v_1\in arc(A_3,A_1)$ or $v_1\in arc(A_1,A_4)$.
First, we suppose that $v_1\in arc(A_3,A_1)$, which implies
$v_3\in arc(A_5,A_3)$. 
Then, $Q$ is on the edge $v_3v_4$, which contradicts $E(\mathcal{P}(v),Q)=1$.
Therefore, we have $v_1\in arc(A_1,A_4)$.
We suppose that $v_1\in arc[u_1,A_4)$.
Then, we have $v_2\in arc[u_2,A_5)$, which implies 
that  $Q$ is on the edge $v_3v_4$. This is a contradiction.
Therefore, we have $v_1\in arc(A_1,u_1)$.
Since $\psi_{\triangle PQR}^5(A_1)=A_1$ and $\psi_{\triangle PQR}^5(v_1)=v_1$,
by the mean value theorem there exists $v'\in arc(A_1,v_1)$ such that
$(\psi_{\triangle PQR}^5)'(v')=1$, which contradicts Lemma \ref{Isosceles-7}.
\end{proof}

Similarly to Lemma \ref{Isosceles-9}, we have following Lemma.

\begin{lem}\label{Isosceles-9-2}
Let $0.8<t<1$.
Then, there does not exist $v\in S^1$  with  $\psi_{\triangle PQR}^5(v)=v$
such that $E(\mathcal{P}(v),P)=1, E(\mathcal{P}(v),Q)=2$, and
$E(\mathcal{P}(v),R)=2$. 
\end{lem}

\begin{lem}\label{Isosceles-10}
Let $0.8<t<1$.
If $\psi_{\triangle PQR}^5(v)=v$ and $\sum_{S\in \{P,Q,R\}}E(\mathcal{P}(v),S)>5$,
there exists $i$ with $1\leq i \leq 5$ such that $v=A_i$.
\end{lem}
\begin{proof}
We suppose that $\psi_{\triangle PQR}^5(v)=v$ and $\sum_{S\in \{P,Q,R\}}E(\mathcal{P}(v),S)>5$.
Then, there exists an edge of $\mathcal{P}(v)$ which includes two vertices of $\triangle PQR$.
If $P$ and $Q$ are included in an edge of $\mathcal{P}(v)$, then it is easy to see that
$v= A_i$ for some $i$ with $1\leq i \leq 5$. 
If $R$ and $Q$ are included in an edge of $\mathcal{P}(v)$, then $u_1$ is a vertex of $\mathcal{P}(v)$. This contradicts Lemma \ref{Isosceles-8}, which states that $\psi_{\triangle PQR}^5(u_1)\ne u_1$.
If $P$ and $R$ are included in an edge of $\mathcal{P}(v)$, we have a contradiction similarly as in the previous case.
\end{proof}

\begin{figure}
    \centering
\begin{tikzpicture}
\draw(0,0) circle (2);

\draw(0,0) circle (2);
\draw (2,0)--(-1.2,1.6)--(0,-2)--(0,2)--(-1.2,-1.6)--cycle;
\draw (-2,0)--(2,0);
\draw (0,-2)--(0,2);
\draw[dashed] (1.89525,0.6387)--(-1.5274,1.291);
\draw[dashed] (0.60433,-1.9065)--(-1.5274,1.291);
\draw[dashed] (0.60433,-1.9065)--(-0.2057203,1.9893916);
\draw[dashed]  (-1.059678,-1.6961959)--(-0.2057203,1.9893916);
\draw[dashed]  (-1.059678,-1.6961959)--(1.9774947,0.2991895871);
\draw[dash dot] (0.60433,1.9065)--(-0.666,0);

\draw (-0.1,0.2)node[left]{$R$};
\draw (-0.05,0.17)node[right]{$O$};
\draw (0,1.2)node[right]{$P$};
\draw (0,-1.2)node[right]{$Q$};
\draw (-1.5,-1.6)node[below]{$A_5$};
\draw (0,2.0)node[above]{$A_4$};
\draw (0,-2.0)node[below]{$A_3$};
\draw (-1.2,1.6)node[above]{$A_2$};
\draw (2,0)node[right]{$A_1$};
\draw (-2,0)node[left]{$B_1$};

\draw (1.89525,0.6387)node[right]{$u_1$};
\draw (-1.5274,1.291)node[left]{$u_2$};
\draw (0.60433,-1.9065)node[below]{$u_3$};
\draw (-0.3557203,1.9893916)node[above]{$u_4$};
\draw (-1.059678,-1.6961959)node[below]{$u_5$};
\draw (1.9774947,0.2991895871)node[right]{$u_6$};
\draw (0.60433,1.9065)node[above]{$w$};

\coordinate (P) at (0,1);
\fill (P) circle [radius=1.5pt];
\coordinate (Q) at (0,-1);
\fill (Q) circle [radius=1.5pt];
\coordinate (R) at (-0.666,0);
\fill (R) circle [radius=1.5pt];

 \end{tikzpicture}
    \caption{Point $w$. }
    \label{fig:pentagon9}
\end{figure}

\begin{lem}\label{Isosceles-11}
Let $0.8<t<1$.
Then, there does not exist $v\in S^1$  with  $\psi_{\triangle PQR}^5(v)=v$
such that $E(\mathcal{P}(v),P)=2, E(\mathcal{P}(v),Q)=2$, and
$E(\mathcal{P}(v),R)=1$. 
\end{lem}
\begin{proof}
We suppose that there exists $v\in S^1$  with  $\psi_{\triangle PQR}^5(v)=v$
such that $E(\mathcal{P}(v),P)=2, E(\mathcal{P}(v),Q)=2$, and
$E(\mathcal{P}(v),R)=1$.
We put $v_i=\psi_{\triangle PQR}^{i-1}(v)$ for $1\leq i \leq 5$.
We suppose that $R$ is on the edge $v_1v_2$ without loss of generality.
We also suppose that $P$ is on the edge $v_2v_3$.
Then, we see that $v_2\in arc[A_3,u_3]$, which implies that 
$v_3\in arc[A_4,u_4]$.
Then, wee that $R$ is on the edge $v_3v_4$, which contradicts $E(\mathcal{P}(v),R)=1$.
Therefore,  $Q$ is on the edge $v_2v_3$.
The fact that $E(\mathcal{P}(v),R)=1$ implies that
$P$ is on the edge $v_3v_4$, $Q$ is on the edge $v_4v_5$, and 
$P$ is on the edge $v_5v_1$. 
Let $w\in {S^1}$ be a point at infinity, which is  an intersection point with the line $RP$ and ${S^1}$
as shown in Figure \ref{fig:pentagon9}.
it is easy to see that $v_3$ is on $arc(A_3,A_4)$.
If $v_3$ is on $arc[A_3,u_1]$, then $R$ is on the edge $v_4v_5$, which contradicts $E(\mathcal{P}(v),R)=1$. 
Therefore, $v_3$ is on $arc(u_1,A_4]$.
If $v_3$ is on $arc[w,A_4]$, then $R$ is on the edge $v_3v_4$, which contradicts $E(\mathcal{P}(v),R)=1$.
Hence, $v_3$ is on $arc(u_1,w)$.
We denote a point symmetric to the $x$-axis with respect to point $v_i$ by $w_i$ for $1\leq i \leq 5$.
Since $\triangle PQR$ is symmetric to the $x$-axis, 
$\psi_{\triangle PQR}^{i}(w_5)=w_{5-i}$ for $1\leq i \leq 4$ 
and $\psi_{\triangle PQR}^{5}(w_5)=w_{5}$.
Then, we see that $E(\mathcal{P}(w_5),P)=2, E(\mathcal{P}(w_5),Q)=2$, 
$E(\mathcal{P}(w_5),R)=1$, and $R$ is on the edge $w_2w_1$.
It is possible to show that $w_5\in arc(u_1,w)$  in the same way as $v_3\in arc(u_1,w)$.
It is not difficult to see that $\psi_{\triangle PQR}^5(u)=\psi_Q\circ\psi_R\circ\psi_P\circ\psi_Q\circ\psi_P(u)$
for $u\in arc[u_1,w]$.
We see that $\psi_P'$ is monotonic increasing in $arc(u_1,w)$, 
$\psi_Q'$ is monotonic increasing in $arc(u_2,A_5)$,
$\psi_P'$ is monotonic increasing in $arc(u_3,A_1)$,
$\psi_R'$ is monotonic increasing in $arc(u_4,A_2)$,
and
$\psi_Q'$ is monotonic increasing in $arc(u_5,A_3)$.
Therefore, by the chain rule we see that 
$(\psi_{\triangle PQR}^5)'$ is monotonic increasing in $arc(u_1,w)$.
We suppose that $v_3\ne w_5$.
First, we assume $w_5\in arc(u_1,v_3)$.
Since $\psi_{\triangle PQR}^5(w_5)=w_5$ and $\psi_{\triangle PQR}^5(v_3)=v_3$,
by the mean value theorem there exists $\alpha\in arc(w_5,v_3)$ such that
$(\psi_{\triangle PQR}^5)'(\alpha)=1$.
On the other hand, from Lemma \ref{Isosceles-8} and the mean value theorem,
there exists $\beta\in arc(u_1,w_5)$ such that
$(\psi_{\triangle PQR}^5)'(\beta)>1$, which contadicts that $(\psi_{\triangle PQR}^5)'$ is monotonic increasing in $arc(u_1,w)$.
Similarly, for the case of $v_3\in arc(u_1,w_5)$ we have a contadiction.
Therefore, we have $v_3=w_5$, which implies that
$v_4=w_4$, $v_5=w_3$, $v_1=w_2$, and $v_2=w_1$. 
Hence, $v_4$ is on the $x$-axis,  which implies that
$v_4=A_1$ or $v_4=B_1$. 
Since $E(\mathcal{P}(v_4),R)\ne E(\mathcal{P}(A_1),R)$, 
we have $v_4=B_1=(-1,0)$.
Since $v_3$ is the intersection point with the line $B_1P$ and ${S^1}$ and
$v_5$ is the intersection point with the line $B_1Q$ and ${S^1}$, 
we have 
\begin{align*}
v_3=\left(\dfrac{1-t^2}{t^2+1}, \dfrac{2t}{t^2+1} \right),
v_5=\left(\dfrac{1-t^2}{t^2+1}, -\dfrac{2t}{t^2+1} \right).
\end{align*}
Similarly, we have
\begin{align*}
v_1=\left(-\dfrac{(t^2-1)^2}{t^4+6t^2+1}, \dfrac{4t(t^2+1)}{t^4+6t^2+1} \right),
v_2=\left(-\dfrac{(t^2-1)^2}{t^4+6t^2+1}, -\dfrac{4t(t^2+1)}{t^4+6t^2+1} \right).
\end{align*}
Since $R$ is on the edge $v_1v_2$, we have 
$R=(-\frac{(t^2-1)^2}{t^4+6t^2+1},0)$, which contradicts
$R=(-\frac{t-1}{t+1},0)$ for $0.8<t<1$.
Thus, we have the desired conclusion.
\end{proof}

\begin{lem}\label{Isosceles-12}
Let $0.8<t<1$.
If for $v\in S^1$ $\psi_{\triangle PQR}^5(v)=v$ holds, then
$v= A_i$ for some $i$ with $1\leq i \leq 5$.
\end{lem}
\begin{proof}
 We suppose that for $v\in S^1$ $\psi_{\triangle PQR}^5(v)=v$.
If $\sum_{S\in \{P,Q,R\}}E(\mathcal{P}(v),S)=5$, from Lemma \ref{Isosceles-9}, Lemma \ref{Isosceles-9-2},
and Lemma \ref{Isosceles-11}, we see that such $v$ does not exist.
 If $\sum_{S\in \{P,Q,R\}}E(\mathcal{P}(v),S)>5$,
by Lemma \ref{Isosceles-10} there exists $i$ with $1\leq i \leq 5$ such that $v=A_i$.
\end{proof}

The following Lemma states that $\psi_{\triangle PQR}^5$ is a contraction map, with the exception of $A_i$ for $1\leq i\leq5$.
\begin{lem}\label{Isosceles-13}
Let $0.8<t<1$.
Let $v\in arc(A_i,\psi_{\triangle PQR}^3(A_i))$ for $1\leq i \leq 5$.
Then, $\psi_{\triangle PQR}^5(v)\in arc(A_i,v)$.
\end{lem}
\begin{proof}
We suppose that $v\in arc(A_1,A_4)$ and $\psi_{\triangle PQR}^5(v)\notin arc(A_1,v)$.
Then, we see that $\psi_{\triangle PQR}^5(v)\in arc[v,A_4)$.
From Lemma \ref{Isosceles-12}, we see that $\psi_{\triangle PQR}^5(v)\in arc(v,A_4)$.
From Lemma \ref{Isosceles-8}, we see $\psi_{\triangle PQR}^5(u_1)\in arc(A_1,u_1)$.
Since $\psi_{\triangle PQR}^5$ is continuous on $arc(A_1,A_4)$, by the intermediate value theorem
there exists $v'\in arc(A_1,A_4)$ such that 
$\psi_{\triangle PQR}^5(v')=v'$, which contradicts Lemma \ref{Isosceles-12}.
The proof for the other cases is analogously done.

\end{proof}

From  Lemma \ref{Isosceles-13}, the next lemma follows immediately.

\begin{lem}\label{Isosceles-14}
Let $0.8<t<1$.
Let $v\in S^1$ and $v\ne A_i$ for any $i$ with $1\leq i \leq 5$.
Then, for $x\in \pi^{-1}(v)$, there exists $\epsilon>0$ such that
$\overline{\psi_{\triangle PQR}}^5(x)=2+x-\epsilon$.
\end{lem}

Using Lemma \ref{Isosceles-14}, we can find an upper bound on the height of the triangle for 
$\rho=\frac{2}{5}$.

\begin{lem}\label{Isosceles-15}
Let $0.8<t<1$.
Let $-1<r'<\frac{t-1}{t+1}$.
We put $R'=(r',0)$.
Then, $\rho(\psi_{\triangle PQR'})<\frac{2}{5}$.
\end{lem}
\begin{proof}
We suppose that $\rho(\psi_{\triangle PQR'})=\frac{2}{5}$.
Then, by Lemma \ref{pentagon-1} there exists $v\in S^1$ such that
$\psi_{\triangle PQR'}^5(v)=v$ and for $x\in \pi^{-1}(v)$, $\overline{\psi_{\triangle PQR'}}^5(x)=2+x$.
We suppose that $v\ne A_i$ for any $i$ with $1\leq i \leq 5$.
Then, from Lemma \ref{Isosceles-14}, 
there exists $\epsilon>0$ such that
$\overline{\psi_{\triangle PQR}}^5(x)=2+x-\epsilon$.
On the other hand,
$\triangle PQR\subset \triangle PQR'$ implies $\overline{\psi_{\triangle PQR'}}(x)\leq\overline{\psi_{\triangle PQR}}(x)$.
Thus, we have a contradiction.
Therefore, we see that $v=A_i$ for some $i$ with $1\leq i \leq 5$.
Then, we have $\{\psi_{\triangle PQR'}^{j-1}(v)\mid 1\leq j \leq 5\}=\{A_j\mid 1\leq j \leq 5\}$.
However, it is easy to see that $\psi_{\triangle PQR'}(A_4)\ne A_j$.
Therefore,  $\rho(\psi_{\triangle PQR'})\ne\frac{2}{5}$.
Since $\triangle PQR\subset \triangle PQR'$, we have $\rho(\psi_{\triangle PQR'})<\frac{2}{5}$.
\end{proof}

\begin{thm}\label{Isosceles-16}
Let $\triangle PQR$ be an isosceles triangle  in $D$ with $d'(P,R)=d'(Q,R)$, $d'(P,Q)>\log 9$ and
$\delta(P,Q,R)>\frac{1}{2}\Delta_1'(P,Q)$.
Then, $\rho(\psi_{\triangle PQR})<\frac{2}{5}$.
\end{thm}
\begin{proof}
Let
\begin{align}
t=\dfrac{e^{d'(P,Q)}-1}{e^{d'(P,Q)}+1},
\end{align}
which implies $0.8<t<1$.
We put $P'=(0,t)$, and $Q'=(0,-t)$.
Then, since $d'(P,Q)=d'(P',Q')$ holds, 
there exist an isometry $\phi$ on $D$
such that $\phi(P)=P'$ and $\phi(Q)=Q'$. 
We put $R'(r,0)=\phi(R)$.
We note that $\delta(P,Q,R)=\delta(P',Q',R')=d'(O,R')$, where 
$O=(0,0)$.
We suppose that $r<0$ without loss of generality.
Since $\delta(P',Q',R')>\frac{1}{2}\Delta_1'(P',Q')$, 
we have $-1<r<\frac{t-1}{t+1}$.
From Lemma \ref{Isosceles-15}, we have $\rho(\psi_{\triangle P'Q'R'})<\frac{2}{5}$.
Since $\rho(\psi_{\triangle PQR})=\rho(\psi_{\triangle P'Q'R'})$, 
we have $\rho(\psi_{\triangle PQR})<\frac{2}{5}$.
\end{proof}

\section{Conjecture}
 
 The following is our conjecture.
 Theorem \ref{pentagon-8} proves that the sufficient condition stated in the conjecture holds.
 Moreover, as seen in the previous chapter, this conjecture holds for a certain type of isosceles triangle.
\begin{con}\label{Con}
Let $\triangle P_1P_2P_3$ be a triangle in $D$.
Then, $\rho(\psi_{\triangle P_1P_2P_3})=\frac{2}{5}$ if and only if 
there exist some $i,j,k$ with $\{i,j,k\}=\{1,2,3\}$ such that
$\Delta_2'(P_i,P_j)\leq \delta(P_i,P_j,P_k)\leq \frac{1}{2}\Delta_1'(P_i,P_j)$
or $\frac{1}{2}\Delta_1'(P_i,P_j) \leq \delta(P_i,P_j,P_k)\leq \Delta_2'(P_i,P_j)$.
\end{con}

\vspace{2cm}

\noindent
Takeo Noda: Faculty of Science, Toho University, JAPAN\\
{\it E-mail address: noda@c.sci.toho-u.ac.jp}\\

\noindent
Shin-ichi Yasutomi: Faculty of Science, Toho University, JAPAN\\
{\it E-mail address: shinichi.yasutomi@sci.toho-u.ac.jp}\\

\noindent
Masamichi Yoshida: Faculty of Science, Osaka Metropolitan University, 3-3-138, Sugimoto, Sumiyoshi, Osaka, 558-
8585, JAPAN\\
{\it E-mail address: masa\_yoshida@omu.ac.jp}

\end{document}